\documentclass[a4paper,11pt,leqno]{article}

\usepackage{a4wide}
\usepackage[english]{babel}
\usepackage[T1]{fontenc}
\usepackage[latin1]{inputenc}
\usepackage{dsfont}
\usepackage{mathrsfs}
\usepackage{amsmath}
\usepackage{amssymb}
\usepackage{verbatim}
\usepackage{amsthm}
\usepackage{graphicx}
\usepackage[usenames]{color}
\usepackage[dvipdf,dvips,rgb,dvipsnames,svgnames,table,x11names]{xcolor}
\usepackage{fancybox}
\usepackage[square,numbers]{natbib}
\usepackage{enumitem}
\usepackage{eurosym}
\def\a{\alpha}
\def\b{\beta}

\def\d{\delta}

\def\e{\varepsilon}

\def\t{\theta}
\def\T{\Theta}

\def\k{\kappa}
\def\l{\lambda}

\def\s{\sigma}

\def\kq{r}
\def\ts{\tilde{S}}
\def\NN{\mathbb N}

\def\RR{\mathbb R}

\def\card{{\rm Card}}
\def\fcar{\mathds{1}}

\def\suchthat{\,|\,}
\def\esp{\mathbf E}
\def\var{{\rm Var}}

\def\prob{\mathbf P}

\def\nzeroun{\mathcal{N}(0,1)}
\theoremstyle{plain}
\newtheorem{theorem}{Theorem}
\newtheorem{lemma}{Lemma}
\newtheorem{proposition}{Proposition}
\newtheorem{corollary}{Corollary}
\newtheorem*{theorem*}{Theorem}
\newtheorem*{proposition*}{Proposition}
\newtheorem*{corollary*}{Corollary}
\theoremstyle{remark}

\newtheorem*{remark*}{Remark}

\newtheorem*{note*}{Note}
\theoremstyle{definition}

\newtheorem*{definition*}{Definition}

\def\sumjd{\sum_{j=1}^d}
\def\hsig{\hat{\s}}

\usepackage{authblk}

\title{More than one Author with different Affiliations}
\author[1,2,3]{Olivier Collier}
\author[4]{La\"etitia Comminges}
\author[5]{Alexandre B. Tsybakov}
\affil[1]{MINES ParisTech}
\affil[2]{Institut Curie}
\affil[3]{INSERM U900}
\affil[4]{Universit\'e Paris-Dauphine}
\affil[5]{CREST-ENSAE}

\begin{document}

%\begin{frontmatter}

\title{Minimax estimation of linear and quadratic functionals on sparsity classes}
%\author{\fnms{Olivier} \snm{Collier}\thanksref{m1} },
%\author{\fnms{La\"etitia} \snm{Comminges} \thanksref{m2}}
%\and 
%\author{\fnms{Alexandre B.} \snm{Tsybakov} \thanksref{m3}}
%\affiliation{MINES ParisTech\thanksmark{m1}, Institut Curie\thanksmark{m1}, INSERM U900\thanksmark{m1},  Universit\'e Paris Dauphine\thanksmark{m2} and CREST-ENSAE\thanksmark{m3}}
\date{}
\maketitle

\begin{abstract}  ~~For the Gaussian sequence model, we obtain non-asymp- totic minimax rates of estimation of the linear, quadratic and the $\ell_2$-norm functionals on classes of sparse vectors and construct optimal estimators that attain these rates. The main object of interest is the class $B_0(s)$ of $s$-sparse vectors $\t=(\t_1,\dots,\t_d)$, for which we also provide completely adaptive estimators (independent of $s$ and of the noise variance $\s$) having only logarithmically slower rates than the minimax ones. Furthermore, we obtain the minimax rates on the $\ell_q$-balls $B_q(r)=\{\t\in \RR^d: \|\t\|_q \le r \}$ where $0<q\le 2$, and $\|\t\|_q= \left(\sum_{i=1}^d | \t_i|^q\right)^{1/q}$.
This analysis shows that there are, in general, three zones in the rates of convergence that we call the sparse zone, the dense zone and the degenerate zone, while a fourth zone appears for estimation of the quadratic functional. We show that, as opposed to estimation~of~$\t$, the correct logarithmic terms in the optimal rates for the sparse zone scale as $\log (d/s^2)$ and not as $\log (d/s)$. For the class $B_0(s)$, the rates of estimation of the linear functional and of the $\ell_2$-norm have a simple elbow at $s=\sqrt{d}$ (boundary between the sparse and the dense zones) and exhibit similar performances, whereas the estimation of the quadratic functional $Q(\theta)$ reveals more complex effects and is not possible only on the basis of sparsity described by the condition $\t\in B_0(s)$.  Finally, we apply our results on estimation of the $\ell_2$-norm to the problem of testing against sparse alternatives. In particular, we obtain a non-asymptotic analog of Ingster-Donoho-Jin theory revealing some effects that were not captured by the previous asymptotic analysis.  
\end{abstract}

%\begin{keyword}
%\kwd{nonasymptotic minimax estimation}
%\kwd{linear functional}
%\kwd{quadratic functional}
%\kwd{sparsity}
%\kwd{unknown noise variance}
%\kwd{thresholding}
%\end{keyword}

%\end{frontmatter}

\section{Introduction}

In this paper, we consider the model
\begin{equation}\label{1}
y_j= \t_j + \s\xi_j, \quad j=1,\dots,d,
\end{equation}
where $\t=(\t_1,\dots,\t_d)\in\RR^d$ is an unknown vector of parameters, $\xi_j$ are i.i.d. standard normal random variables, and $\s>0$ is the noise level. We study the problem of estimation of linear and quadratic functionals 
$$
L(\t)=\sum_{i=1}^d \t_i, \qquad \text{and} \qquad Q(\t)=\sum_{i=1}^d \t_i^2,
$$
and of the $\ell_2$-norm
$$
\|\t\|_2 = \sqrt{Q(\t)}
$$
based on the observations $y_1,\dots,y_d$.

In this paper, we assume that $\t$ belongs to a given subset $\Theta$ of $\RR^d$. We will be considering classes $\Theta$ with elements satisfying the sparsity  constraints $\|\t\|_0 \le s$ where $\|\t\|_0$ denotes the number of non-zero components of $\t$, or $\|\t\|_q \le r $   where $$\|\t\|_q= \left(\sum_{i=1}^d | \t_i|^q\right)^{1/q}.$$ Here,  $r,q>0$ and the integer $s\in [1,d]$ are given constants. 

Let $T(\t)$ be one of the functionals $L(\t)$, $Q(\t)$ or $ \sqrt{Q(\t)}$. As a measure of quality of an estimator $\hat T$ of the functional $T(\t)$, we consider the maximum squared risk   
\begin{equation*}
\sup_{\t\in \Theta} \esp_\t(\hat{T} - T(\theta))^2,
\end{equation*}
where $\esp_\t$ denotes the expectation with respect to the probability measure $\prob_\t$ of the vector of observations $(y_1,\dots,y_d)$ satisfying \eqref{1}. The best possible quality is characterized by the minimax risk
$$
R^*_T(\Theta) = \inf_{\hat T} \sup_{\t\in \Theta} \esp_\t(\hat{T} - T(\theta))^2,
$$
where $\inf_{\hat T}$ denotes the infimum over all estimators. In this paper, we find minimax optimal estimators of $T(\t)$, i.e., estimators $\tilde T$ such that 
\begin{equation}\label{2}
\sup_{\t\in \Theta} \esp_\t(\tilde{T} - T(\theta))^2 \asymp R^*_T(\Theta).
\end{equation}
Here and below, we write $a\asymp b$ if $c\le a/b\le C$ for some absolute positive constants $c$ and $C$. Note that the minimax optimality is considered here in the non-asymptotic sense, i.e., \eqref{2} should hold for all $d$ and $ \s$.

\smallskip

The literature on minimax estimation of linear and quadratic functionals is rather extensive. The analysis of estimators of linear functionals from the minimax point of view was initiated in \cite{IbragimovHasminskii1984} while for the quadratic functionals we refer to \cite{DonohoNussbaum1990}.  These papers, as well as the subsequent publications \cite{CaiLow2005a,CaiLow2005,DonohoLiu1991, EfromovichLow1996, Golubev2004, GolubevLevit2004, Johnstone2001a, Johnstone2001b, JuditskyNemirovski2009, Klemela2006, KlemelaTsybakov2001,LaurentLudenaPrieur2008, LaurentMassart2000, LepskiNemirovskiSpokoiny1999, Nemirovski2000},  focus on minimax estimation of functionals on the classes $\Theta$ describing the smoothness properties of functions in terms of their Fourier or wavelet coefficients. Typical examples are Sobolev ellipsoids, hyperrectangles or Besov bodies while a typical example of linear functional is the value of a smooth function at a point. In this framework, a deep analysis of estimation of functionals is now available including the minimax rates (and in some cases the minimax constants), oracle inequalities and adaptation. Extensions to linear inverse problems have been considered in detail by \cite{Butucea2007,ButuceaComte2009,GoldenshlugerPereverzev2003}. Note that classes $\Theta$ studied in this literature are convex classes. Estimation of functionals on the non-convex sparsity classes $B_0(s)=\{\t\in \RR^d: \|\t\|_0 \le s \}$ or $B_q(r)=\{\t\in \RR^d: \|\t\|_q \le r \}$ with $0<q< 1$ has received much less attention. We are only aware of the paper \cite{CaiLow2004}, which establishes upper and lower bounds on the minimax risk for estimators of the linear functional $L(\t)$ on the class
$B_0(s)$. However,  that paper considers the special case when $s<d^a$ for some $a<1/2$, and  $\sigma = 1/\sqrt{d}$ and there is a logarithmic gap between the upper and lower bounds.  Minimax rates for the estimation of $Q(\t)$ and of the $\ell_2$-norm on the classes $B_0(s)$ and $B_q(r)$, $0<q<2$, were not studied. Note, that estimation the $\ell_2$-norm is closely related to minimax optimal testing of hypotheses under the $\ell_2$ separation distance in the spirit of~\cite{IngsterSuslina2003}. Indeed, the optimal tests for this problem are based on estimators of the $\ell_2$-norm.   A non-asymptotic study of minimax rates of testing for the classes $B_0(s)$ and $B_q(r)$, $0<q<2$, is given in \cite{Baraud2002} and \cite{Verzelen2012}. But for the testing problem, the risk function is different and these papers do not provide results on the estimation of the $\ell_2$-norm. Note also that the upper bounds on the minimax rates of testing in \cite{Baraud2002} and \cite{Verzelen2012} depart from the lower bounds by a logarithmic factor.

\smallskip

In this paper, we find non-asymptotic minimax rates of estimation of the above three functionals on the sparsity classes $B_0(s)$, $B_q(r)$ and construct optimal estimators that attain these rates. We deal with non-convex classes $B_q$ ($0<q<1$) for the linear functional and with the classes that are not quadratically convex ($0<q<2$) for $Q(\t)$ and of the $\ell_2$-norm. Our main object of interest is the class $B_0(s)$, for which we also provide completely adaptive estimators (independent of $\s$ and $s$) having only logarithmically slower rates than the minimax ones. Some interesting effects should be noted. First, we show that, for the linear functional and the $\ell_2$-norm there are, in general, three zones in the rates of convergence that we call the sparse zone, the dense zone and the degenerate zone, while for the quadratic functional an additional fourth zone appears. Next, as opposed to estimation of the vector $\t$ in the $\ell_2$-norm, cf. \cite{DonohoJohnstone1994,BirgeMassart2001,AbramovichGrinshtein2010,Johnstone,RigolletTsybakov2011,Verzelen2012}, the correct logarithmic terms in the optimal rates for the sparse zone scale as $\log (d/
 s^2)$ and not as $\log (d/s)$. Noteworthy, for the class $B_0(s)$, the rates of estimation of the linear functional and of the $\ell_2$-norm have a simple elbow at $s=\sqrt{d}$ (boundary between the sparse and the dense zones) and exhibit similar performances, whereas the estimation of the quadratic functional $Q(\theta)$ reveals more complex effects and is not possible only on the basis of sparsity described by the condition $\t\in B_0(s)$.  Finally, we apply our results on estimation of the $\ell_2$-norm to the problem of testing against sparse alternatives. In particular, we obtain a non-asymptotic analog of Ingster-Donoho-Jin theory revealing some effects that were not captured by the previous asymptotic analysis.

\smallskip

\section{Minimax estimation of the linear functional}

In this section, we study the minimax rates of estimation of the linear functional $L(\theta)$ and we construct minimax optimal estimators. 

Assume first that $\Theta$ is the class of $s$-sparse vectors
 $B_0(s)=\{\t\in \RR^d: \|\t\|_0 \le s \}$ where $s$  is a given integer,  $1\le s\le d$.  Consider the estimator
  $$
  \hat L = \left\{
  \begin{array}{lcl}
  \sum_{j=1}^d y_j ~\fcar_{\{ |y_j|>\s\sqrt{2\log (1+d/s^2)} \}}& \text{if}& s<\sqrt{d},\\
  \sum_{j=1}^d y_j& \text{if} & s\ge \sqrt{d},
  \end{array}
  \right.
 $$
 where $\fcar_{\{\cdot\}}$ denotes the indicator function. 

 The following theorem shows that
 $$
 \psi_\s^L(s,d) = \s^2 s^2 \log (1+d/s^2)
 $$
 is the minimax rate of estimation of the linear functional on the class $B_0(s)$ and that $\hat{L}$ is a minimax optimal estimator.
 \begin{theorem}\label{t_linear} There exist absolute constants $c>0, C>0$ such that, for any integers $s,d$ satisfying $1\le s\le d$, and any $\s>0$,
 \begin{equation}\label{eq:t_linear1}
\sup_{\t\in B_0(s)} \esp_\t(\hat{L} - L(\theta))^2 \le C \psi_\s^L(s,d),
\end{equation}
and 
 \begin{equation}\label{eq:t_linear2}
R^*_L(B_0(s))  \ge c \psi_\s^L(s,d).
\end{equation}
 \end{theorem}
 Proofs of  \eqref{eq:t_linear1} and of \eqref{eq:t_linear2} are given in Sections~\ref{sec:proof:upper} and~\ref{sec:proof:lower} respectively.
Note that since $\log(1+u)\ge u/2$ for $0<u\le 1,$ and $\log(1+u)\le u$ we have 
\begin{equation}\label{eq:t_linear3}
\s^2 s^2 \log (1+d/s^2) \asymp \min(\s^2 s^2 \log (1+d/s^2),    \s^2d)  
\end{equation} 
for all $1\le s\le d$. Thus,
\begin{equation}\label{eq:t_linear4}
R^*_L(B_0(s)) \asymp \min(\s^2 s^2 \log (1+d/s^2),    \s^2d)  .
\end{equation} 

\smallskip

We consider now the classes $ B_q(\kq)=\{\t\in \RR^d: \|\t\|_q \le \kq \}$, where $0<q\le 1$, and $ \kq$ is a  positive number. For any $\kq,\s,q>0$ any integer $d\ge1$, we define the integer
\begin{equation}\label{def:m}
m=\max\{s\ge 1: \s^2\log (1+d/s^2) \le r^2s^{-2/q}, s\in  \NN\}
\end{equation}
if the set $\{s\ge 1: \s^2\log (1+d/s^2) \le r^2s^{-2/q}, s\in  \NN\}$  is non-empty, and we put $m=0$ if this set is empty. The next two theorems show that the optimal rate of convergence of estimators of the linear functional on the class $ B_q(\kq)$ is of the form:
$$
 \psi^L_{\s,q}(\kq, d) =\left\{
  \begin{array}{lcl}
 \s^2 m^2 \log (1+d/m^2) & \text{if}& m \ge 1,\\
 \kq^2 & \text{if}& m=0.
   \end{array}
  \right.
$$
The following theorem shows that 
 $   \psi^L_{\s,q}(\kq, d) $
is a lower bound on the convergence rate of the minimax risk of the linear functional on the class $ B_q(\kq)$.
\begin{theorem}\label{t_lin_gamma_lower}
If $0<q\le1$, then there exists a constant $c>0$ such that, for any integer $d\ge1$ and any $\kq,\s>0$, we have
 \begin{equation}\label{eq:t_lin1}
R^*_L( B_q(\kq))  \ge c  \psi^L_{\s,q}(\kq, d).
\end{equation}
\end{theorem}
The proof of Theorem \ref{t_lin_gamma_lower} is given in Section \ref{sec:proof:lower}.

We now turn to the construction of minimax optimal estimators on $ B_q(\kq)$.
For $0<q\le 1$, define the following statistic
$$
  \hat L_q = \left\{
  \begin{array}{lcl}
  \sum_{j=1}^d y_j & \text{if}& m > \sqrt{d},\\
  \sum_{j=1}^d y_j~\fcar_{\{ |y_j|>2\s\sqrt{2\log (1+d/m^2)} \}}  & \text{if}& 1\le m \le \sqrt{d},\\
  0 & \text{if} & m=0.
  \end{array}
  \right.
 $$

\begin{theorem}\label{t_lin_gamma_upper} Let $0<q\le 1$. There exists a constant $C>0$ such that, for any integer $d\ge1$ and any $\kq,\s>0$, we have
\begin{equation}\label{eq:t_lin_gamma2}
\sup_{\t\in  B_q(\kq)} \esp_\t(\hat{L}_q - L(\theta))^2 \le C  \psi^L_{\s,q}(\kq, d).
\end{equation}
\end{theorem}
The proof of Theorem \ref{t_lin_gamma_upper} is given in Section \ref{sec:proof:upper}.
Theorems~\ref{t_lin_gamma_lower}  and \ref{t_lin_gamma_upper}  imply that $ \psi^L_{\s,q}(\kq, d)$ is the minimax rate of estimation of the linear functional on the ball $B_q(\kq)$ and that $\hat{L}_q$ is a minimax optimal estimator.

Some remarks are in order here. Apart from the degenerate case $m=0$ when the zero estimator is optimal, we obtain on $B_q(\kq)$ the same expression for the optimal rate as on the class $B_0(s)$, with the difference that the sparsity $s$ is now replaced by the ``effective sparsity``~$m$. Heuristically, $m$ is obtained as a solution of 
$$
\s^2 m^2 \log (1+d/m^2) \asymp r^2 m^{2-2/q}
$$
where the left hand side represents the estimation error for $m$-sparse signals established in Theorem~\ref{t_linear} and the right hand side gives the error of approximating a vector from $B_q(\kq)$ by an $m$-sparse vector in squared $\ell_1$-norm.
Note also that, in view of \eqref{eq:t_linear3}, we can equivalently write the optimal rate in the form 
$$
 \psi^L_{\s,q}(\kq, d) \asymp \left\{
  \begin{array}{lcl}
\s^2d & \text{if} & m > \sqrt{d},\\
 \s^2 m^2 \log (1+d/m^2) & \text{if}& 1\le m \le \sqrt{d},\\
  \kq^2 & \text{if}& m=0.
   \end{array}
  \right.
$$
Thus, the optimal rate on $ B_q(\kq)$ has in fact three regimes that we will call the dense zone ($ m > \sqrt{d}$), the sparse zone ($1\le m \le \sqrt{d}$), and the degenerate zone ($m=0$). 
Furthermore, it follows from the definition of $m$ that the rate $ \psi^L_{\s,q}(\kq, d)$ in the sparse zone is of the order  
$
\s^2(r/\s)^{2q} \log^{1-q}(1+d (\s/r)^{2q}), 
$
which leads to 
$$
 \psi^L_{\s,q}(\kq, d) \asymp \left\{
  \begin{array}{lcl}
\s^2d & \text{if} & m > \sqrt{d},\\
 \s^2(r/\s)^{2q} \log^{1-q}(1+d (\s/r)^{2q}) & \text{if}& 1\le m \le \sqrt{d},\\
  \kq^2 & \text{if}& m=0.
   \end{array}
  \right.
$$
 In particular, for $q=1$, the logarithmic factor disappears from the rate, and the optimal rates in the sparse and degenerate zones are both equal to $r^2$. Therefore, for $q=1$, there is no need to introduce thresholding in the definition of $\hat L_q$, and it is enough to use only the zero estimator for $m \le \sqrt{d}$ and the estimator $\sum_{j=1}^d y_j$ for $m > \sqrt{d}$ to achieve the optimal rate.  
% This behavior extends to $q>1$ as shown in the next theorem.
% %%%%%%%%
% \begin{theorem}\label{t_lin_q_geq_1} Let $q\ge 1$. 
% Then \eqref{eq:t_lin1} and \eqref{eq:t_lin_gamma2} are valid with $\hat L_q$ and $\psi^L_{\s,q}(\kq, d)$ defined as
% $$
%  \hat L_q = \left\{
%  \begin{array}{lcl}
%  \sum_{j=1}^d y_j & \text{if}& m > \sqrt{d},\\
%  0  & \text{if}& 0\le m \le \sqrt{d},
%  \end{array}
%  \right.
% $$
%and
% $$
% \psi^L_{\s,q}(\kq, d) = \left\{
%  \begin{array}{lcl}
%\s^2d & \text{if} & m > \sqrt{d},\\
%\kq^2  & \text{if}& 0\le m \le \sqrt{d}.
%   \end{array}
%  \right.
%$$
%\end{theorem}

\smallskip

\section{Minimax estimation of the quadratic functional}

Consider now the problem of estimation of the quadratic functional $Q(\t)=\sum_{i=1}^d \t_i^2$. For any integers $s,d$ satisfying $1\le s\le d$, and any $\s>0$, we introduce the notation
$$
\bar \psi_\s(s,d) =\left\{
  \begin{array}{lcl}
 \s^4 s^2\log^2 (1+d/s^2) & \text{if}& s<\sqrt{d},\\
  \s^4d & \text{if} & s\ge \sqrt{d}. 
   \end{array}
  \right.
$$
 The following theorem shows that
 $$
  \psi_\s^Q(s,d,\k) = \min\{\k^4, \max\{\s^2\k^2, \bar \psi_\s(s,d) \}\}
 $$
is a lower bound on the convergence rate of the minimax risk of the quadratic functional on the class $B_2(\k)\cap B_0(s)$, where $B_2(\k)=\{\t\in \RR^d: \, \|\t\|_2\le \k\}$. 
%%%%%%%%%%%%%
 \begin{theorem}\label{t_quadr_lower} There exists an absolute constant $c>0$ such that, for any integers $s,d$ satisfying $1\le s\le d$, and any $\k,\s>0$, we have
 \begin{equation}\label{eq:t_quadr1}
R^*_Q(B_2(\k)\cap B_0(s))  \ge c \,\psi_\s^Q(s,d,\k).
\end{equation}
 \end{theorem}
 The proof of Theorem \ref{t_quadr_lower} is given in Section \ref{sec:proof:lower}.
 
One of the consequences of Theorem~\ref{t_quadr_lower}  is that $R^*_Q(B_0(s))= \infty$ (set $\k=\infty$ in \eqref{eq:t_quadr1}). Thus, only smaller classes than $B_0(s)$ are of interest when estimating the quadratic functional. The class $B_2(\k)\cap B_0(s)$ naturally arises in this context but other classes can be considered as well.

We now turn to the construction of minimax optimal estimator on $B_2(\k)\cap B_0(s)$. Set 
\begin{equation*}
\a_s = \esp\big(X^2\suchthat X^2>2\log(1+d/s^2)\big)  = \frac{\esp\big(X^2\fcar_{\{|X|>\sqrt{2\log(1+d/s^2)}\}}\big)}{\prob\big(|X|>\sqrt{2\log(1+d/s^2)}\,\big)}, 
\end{equation*}
where $X\sim\nzeroun$ denotes the standard normal random variable. Introduce the notation
$$
 \psi_\s(s,d,\k) = \max\{\s^2\k^2, \bar \psi_\s(s,d)\}.
$$
Thus,
\begin{equation}\label{eq:psi}
 \psi_\s^Q(s,d,\k) = \min\{\k^4,  \psi_\s(s,d,\k)\}.
\end{equation}
Define the following statistic
$$
  \hat Q = \left\{
  \begin{array}{lcl}
  \sum_{j=1}^d (y_j^2 -\a_s \s^2)~\fcar_{\{ |y_j|>\s\sqrt{2\log (1+d/s^2)} \}}& \text{if}& s<\sqrt{d} \ \text{and}\  \k^4 \ge  \psi_\s(s,d,\k),\\
  \sum_{j=1}^d y_j^2 -d\s^2\phantom{~\fcar_{\{ |y_j|>\s\sqrt{2\log (1+d/s^2)} \}}}& \text{if} & s\ge \sqrt{d} \ \text{and}  \  \k^4 \ge  \psi_\s(s,d,\k),\\
  0 & \text{if} & 
  \k^4 <  \psi_\s(s,d,\k).
  \end{array}
  \right.
 $$
  \begin{theorem}\label{t_quadr_upper} There exists an absolute constant $C>0$ such that, for any integers $s,d$ satisfying $1\le s\le d$, and any $\k,\s>0$, we have
  \begin{equation}\label{eq:t_quadr2}
\sup_{\t\in B_2(\k)\cap B_0(s)} \esp_\t(\hat{Q} - Q(\theta))^2 \le C\, \psi_\s^Q(s,d,\k).
\end{equation}
 \end{theorem}
The proof of Theorem \ref{t_quadr_upper} is given in Section \ref{sec:proof:upper}.
Theorems~\ref{t_quadr_lower}  and \ref{t_quadr_upper}  imply that $\psi_\s^Q(s,d,\k)$ is the minimax rate of estimation of the quadratic functional on the class $ B_2(\k)\cap B_0(s)$ and that $\hat{Q}$ is a minimax optimal estimator.

As a corollary, we obtain the minimax rate of convergence on the class $ B_2(\k)$ (set $s=d$ in Theorems~\ref{t_quadr_lower}  and~\ref{t_quadr_upper}). In this case, the estimator $\hat Q$ takes the form
$$
  \hat Q_* = \left\{
  \begin{array}{lcl}
  \sum_{j=1}^d y_j^2 -d\s^2& \text{if} &   \k^4 \ge   \max\{\s^2\k^2,\s^4d\},\\
  0 & \text{if} & 
  \k^4 <  \max\{\s^2\k^2,\s^4d\}.  \end{array}
  \right.
 $$
  \begin{corollary}\label{cor_quadr_full} There exist absolute constants $c,C>0$ such that, for any $\k,\s>0$, we have
  \begin{equation}\label{eq:cor_quadr1}
\sup_{\t\in B_2(\k)} \esp_\t(\hat{Q}_* - Q(\theta))^2 \le C \min\{\k^4, \max(\s^2\k^2,\s^4d)\},
\end{equation}
and 
\begin{equation}\label{eq:cor_quadr2}
R^*_Q(B_2(\k))  \ge c  \min\{\k^4, \max(\s^2\k^2,\s^4d)\}.
\end{equation}
 \end{corollary}
 Note that the upper bounds of Theorem~\ref{t_quadr_upper} and Corollary~\ref{cor_quadr_full} obviously remain valid for the positive part estimators $\hat Q_+=\max\{\hat Q,0\}$, and $\hat Q_{*,+}= \max\{\hat Q_*,0\}$. The upper rate as in \eqref{eq:cor_quadr1} on the class $B_2(\k)$
 with an extra logarithmic factor is obtained for different estimators in \cite{Johnstone2001a, Johnstone2001b}.
%%%%%%%%%%%%%%

\smallskip

Alternatively, we consider the classes $B_q(\kq)$, where $\kq$ is a positive number  and $0<q<2$. As opposed to the case of $B_0(s)$, we do not need to consider intersection with $B_2(\k)$. Indeed, it is granted that the $\ell_2$-norm of $\t$ is uniformly bounded thanks to the inclusion $B_q(r)\subseteq B_2(r)$.
 For any  $r,\s>0$, $0<q<2$, and any integer $d\ge1$  we set 
$$
\psi_{\s,q}^Q(r,d) =\left\{
 \begin{array}{lcl}
 \max\{\s^2 r^2,  \s^4d\} & \text{if} & m > \sqrt{d},\\
  \max\{\s^2 r^2,  \s^4 m^2 \log^2 (1+d/m^2)\} & \text{if}& 1\le m \le \sqrt{d},\\
  \kq^4 & \text{if}& m=0,
   \end{array}
   \right.
$$
where $m$ is the integer defined above (cf.~\eqref{def:m}) and depending only on $d,r,\s,q$.
The following theorem shows that $\psi_{\s,q}^Q(r,d)$
is a lower bound on the convergence rate of the minimax risk of the quadratic functional on the class $B_q(\kq)$.

\begin{theorem}\label{t_quadr_gamma_lower}
Let $0<q<2$. There exists a constant $c>0$ such that, for any integer $d\ge1$, and any $r,\s>0$, we have
  \begin{equation}\label{eq:t_quadr3}
R^*_Q(B_q(\kq))  \ge c \,\psi_{\s,q}^Q(r,d).
\end{equation}
\end{theorem}

We now turn to the construction of minimax optimal estimators on $B_q(\kq)$. Consider the following statistic
$$
  \hat Q_q = \left\{
  \begin{array}{lcl}
   \sum_{j=1}^d y_j^2 -d\s^2 & \text{if}&
  m > \sqrt{d},\\
\sum_{j=1}^d (y_j^2 - \tilde \a_m \s^2)~\fcar_{\{ |y_j|>2\s\sqrt{2\log (1+d/m^2)} \}}& \text{if} & 1\le m \le \sqrt{d},\\
  0 & \text{if} & m=0,
  \end{array}
  \right.
 $$
where
$ \tilde \a_m = \esp\big(X^2\suchthat X^2>8\log (1+d/m^2)\big),\, \ X\sim\nzeroun.$

\begin{theorem}\label{t_quadr_gamma_upper} Let $0<q<2$. There exists a constant $C>0$ such that, for any integer $d\ge1$ , and any $r,\s>0$, we have
\begin{equation}\label{eq:t_quadr_gamma2}
\sup_{\t\in B_q(\kq)} \esp_\t(\hat{Q}_q - Q(\theta))^2 \le C \psi_{\s,q}^Q(r, d).
\end{equation}
\end{theorem}
The proof of Theorem \ref{t_quadr_gamma_upper} is given in Section \ref{sec:proof:upper}.
Theorems~\ref{t_quadr_gamma_lower}  and \ref{t_quadr_gamma_upper}  imply that $\psi_{\s,q}^Q(r,d)$ is the minimax rate of estimation of the quadratic functional on the class $B_q(\kq)$ and that $\hat{Q}_q$ is a minimax optimal estimator.

Notice that, in view of the definition of $m$,  in the sparse zone we have  
$$
\s^4 m^2 \log^2 (1+d/m^2) \asymp \s^4(r/\s)^{2q} \log^{2-q}(1+d (\s/r)^{2q}), 
$$
which leads to 
$$
 \psi^Q_{\s,q}(\kq, d) \asymp \left\{
   \begin{array}{lcl}
 \max\{\s^2 r^2,  \s^4d\} & \text{if} & m > \sqrt{d},\\
  \max\{\s^2 r^2, \s^4(r/\s)^{2q} \log^{2-q}(1+d (\s/r)^{2q}) \} & \text{if}& 1\le m \le \sqrt{d},\\
  \kq^4 & \text{if}& m=0.
   \end{array}
     \right.
$$
One can check that for $q=2$ this rate is of the same order as the rate obtained in Corollary~\ref{cor_quadr_full}.

\smallskip

\section{Minimax estimation of the $\ell_2$-norm}\label{sec:l2norm}

Interestingly, the minimax rates of estimation of the $\ell_2$-norm $\|\t\|_2 = \sqrt{Q(\t)}$ do not depend on the radius $\k$, as opposed to the rates for $Q(\t)$ established above. It turns out that the restriction to $B_2(\k)$ is not needed to get meaningful results for estimation of $\sqrt{Q(\t)}$ on the sparsity classes.  We drop this restriction and assume that $\T= B_0(s)$. Consider the estimator 
$$
\hat N =  \sqrt{\max\{\hat Q_{\bullet}, 0\}}
$$
where
$$
  \hat Q_{\bullet} = \left\{
  \begin{array}{lcl}
  \sum_{j=1}^d (y_j^2 -\a_s \s^2)~\fcar_{\{ |y_j|>\s\sqrt{2\log (1+d/s^2)} \}}& \text{if}& s<\sqrt{d}, \\
  \sum_{j=1}^d y_j^2 -d\s^2& \text{if} & s\ge \sqrt{d}.
  \end{array}
  \right.
 $$
The following theorem shows that $\hat N$ is a minimax optimal estimator of the $\ell_2$-norm $\|\t\|_2 = \sqrt{Q(\t)}$ on the class $B_0(s)$ and that the corresponding minimax rate of convergence is 
$$
\psi_\s^{\sqrt{Q}}(s,d) =\left\{
  \begin{array}{lcl}
 \s^2 s\log (1+d/s^2) & \text{if}& s<\sqrt{d},\\
  \s^2\sqrt{d} & \text{if} & s\ge \sqrt{d}. 
   \end{array}
  \right.
$$
 \begin{theorem}\label{t_l2norm} There exist absolute constants $c>0, C>0$ such that, for any integers $s,d$ satisfying $1\le s\le d$, and any $\s>0$,
 \begin{equation}\label{eq:t_l2norm1}
\sup_{\t\in B_0(s)} \esp_\t(\hat{N} - \|\t\|_2)^2 \le C \psi_\s^{\sqrt{Q}}(s,d),
\end{equation}
and 
 \begin{equation}\label{eq:t_l2norm2}
R^*_{\sqrt{Q}}(B_0(s))  \ge c \psi_\s^{\sqrt{Q}}(s,d).
\end{equation}
 \end{theorem}
 Proofs of  \eqref{eq:t_l2norm1} and of \eqref{eq:t_l2norm2} are given in Sections~\ref{sec:proof:upper} and~\ref{sec:proof:lower} respectively.

\smallskip

Our next step is to analyze the classes $B_q(r)$. For any  $r,\s>0$, $0<q<2$, and any integer $d\ge1$  we set 
$$
\psi_{\s,q}^{\sqrt{Q}}(r,d) =\left\{
 \begin{array}{lcl}
 \s^2\sqrt{d} & \text{if} & m > \sqrt{d},\\
  \s^2 m \log (1+d/m^2) & \text{if}& 1\le m \le \sqrt{d},\\
  \kq^2 & \text{if}& m=0,
   \end{array}
   \right.
$$
where $m$ is the integer defined above (cf.~\eqref{def:m}) and depending only on $d,r,\s,q$. The estimator that we consider when $\t$ belongs to the class $B_q(r)$ is
$$
\hat{N}_q =  \sqrt{\max\{{\hat Q}_q, 0\}}.
$$
\begin{theorem}\label{t_l2norm_Lqclass}
Let $0<q<2$. There exist constants $C, c>0$ such that, for any integer $d\ge1$, and any $r,\s>0$, we have
  \begin{equation}\label{eq:t_l2norm3}
  \sup_{\t\in B_q(r)} \esp_\t(\hat{N}_q - \|\t\|_2)^2 \le C \psi_{\s,q}^{\sqrt{Q}}(r,d),
\end{equation}
and 
 \begin{equation}\label{eq:t_l2norm4}
R^*_{\sqrt{Q}}(B_q(r))  \ge c \psi_{\s,q}^{\sqrt{Q}}(r,d).
\end{equation}
\end{theorem}
Proofs of  \eqref{eq:t_l2norm3} and of \eqref{eq:t_l2norm4} are given in Sections~\ref{sec:proof:upper} and~\ref{sec:proof:lower} respectively.

\smallskip

As in the case of linear and quadratic functionals, we have an equivalent expression for the optimal rate:
$$
 \psi_{\s,q}^{\sqrt{Q}}(r, d) \asymp \left\{
  \begin{array}{lcl}
\s^2\sqrt{d} & \text{if} & m > \sqrt{d},\\
 \s^2(r/\s)^{q} \log^{1-q/2}(1+d (\s/r)^{2q}) & \text{if}& 1\le m \le \sqrt{d},\\
  r^2 & \text{if}& m=0.
   \end{array}
  \right.
$$
Though we formally did not consider the case $q=2$, note that the logarithmic factor disappears from the above expression when $q=2$, and the optimal rates in the sparse and degenerate zones are both equal to $r^2$. This suggests that, for $q=2$, there is no need to introduce thresholding in the definition of $\hat N_q$, and it is enough to use only the zero estimator for $m \le \sqrt{d}$ and the estimator $\big(\max\big\{\sum_{j=1}^d y_j^2-d\s^2, 0\big\}\big)^{1/2}$ for $m > \sqrt{d}$ to achieve the optimal rate.  

\smallskip

\section{Estimation with unknown noise level}\label{estimation_unknown_noise}

\smallskip

In this section, we discuss modifications of the above estimators when the noise level $\s$ is unknown. A general idea leading to our construction is that the smallest $y_j^2$ are likely to correspond to zero components of $\t$, and thus to contain information on $\s$ not corrupted by~$\t$. Here, we will demonstrate this idea only for estimation of $s$-sparse vectors in the case $s\leq \sqrt{d}$. Then, not more than $d-\sqrt{d}$ smallest $y_j^2$ can be used for estimation of the variance. Throughout this section, we assume that $d\ge 3$.

We start by considering estimation of the linear functional. Then it is enough to replace $\s$ in the definition of $\hat L$ by
the following statistic 
$$\hsig=3\Big(\frac1d \sum_{j\le d-\sqrt{d}} y_{(j)}^2\Big)^{1/2}$$
where $y_{(j)}^2\le \dots \le y_{(d)}^2$ are the order statistics associated to $y_1^2,\dots,y_d^2$. Note that $\hsig$ is not a good estimator of $\s$ but rather an over-estimator. The resulting estimator of $L(\t)$ is 
$$\tilde L=\sum_{j=1}^{d}y_j ~\fcar_{\{ |y_j|>\hsig\sqrt{2\log (1+d/s^2)} \}}.$$
%%%%
 \begin{theorem}\label{sigma_lin}
  There exists an absolute constant $C$ such that, for any integers $s$ and $d$ satisfying  $s\leq \sqrt{d}$,  and any $\sigma>0$,
  $$
  \sup_{\t\in B_0(s)} \esp_\t(\tilde{L} - L(\theta))^2 \le C \psi_\s^L(s,d).
  $$
  \end{theorem}
The proof of Theorem~\ref{sigma_lin} is given in Section~\ref{sec:proof:upper}.

\smallskip

Note that the estimator $\tilde L$ depends on $s$. To turn it into a completely data-driven one, we may consider 
$$\tilde L'=\sum_{j=1}^{d}y_j ~\fcar_{\{ |y_j|>\hsig\sqrt{2\log d} \}}.$$
Inspection of the proof of Theorem~\ref{sigma_lin} leads to the conclusion that 
\begin{equation}  
\sup_{\t\in B_0(s)} \esp_\t(\tilde{L}' - L(\theta))^2 \le C \s^2 s^2 \log d. 
\end{equation}
Thus, the rate for the data-driven estimator $\tilde L'$ is not optimal but the deterioration is only in the expression under the logarithm.

A data-driven estimator of the quadratic functional can  be taken in the form:
\begin{equation*}
 \tilde{Q}= \sum_{j=1}^d 
 y^2_j 
 ~\fcar_{\{|y_j|>\hat\s\sqrt{2\log d}\}}.
\end{equation*}

The following theorem shows that the estimator $\tilde{Q}$ is nearly minimax on $B_2(\k)\cap B_0(s)$ for $s\le \sqrt{d}$.

\begin{theorem}\label{theorem_upper_bound_unknown_noise}
There exists an absolute constant $C$ such that,  for any integers $s$ and $d$ satisfying  $s\leq \sqrt{d}$,  and any $\sigma>0$,\begin{equation*}
\sup_{\t\in B_2(\k)\cap B_0(s)} \esp_\t(\tilde{Q} - Q(\t))^2 \leq C \max\Big\{\s^2\k^2,\s^4s^2\log^2d\Big\}.
\end{equation*}
\end{theorem}

The proof of Theorem~\ref{theorem_upper_bound_unknown_noise} is given in  Section~\ref{sec:proof:upper}.
%%%%%%%%%%%%%

\smallskip

\section{Consequences for the problem of testing}\label{sec:tests}

The results on estimation of the $\ell_2$-norm stated above allow us to obtain the solution of the problem of non-asymptotic minimax testing on the classes $B_0(s)$ and $B_q(r)$ under the  $\ell_2$ separation distance. For $q\ge 0$, $u>0$, and $\d>0$, consider the set 
$$
\Theta_{q,u}(\d)= \{\t \in B_q(u): \ \|\t\|_2 \ge \d\}.
$$
Assume that we wish to test the hypothesis ${\bf H}_0: \t =0$ against the alternative
$$
{\bf H}_1:  \  \t\in \Theta_{q,u}(\d).
$$
Let $\Delta$ be a test statistic with values in $\{0,1\}$. 
We define the risk of test $\Delta$ as the sum of the first type error and the maximum second type error:
$$
\prob_0(\Delta=1) + \sup_{\t\in \Theta_{q,u}(\d)}\prob_\t(\Delta=0).
$$
A benchmark value is the minimax risk of testing
$$
{\mathcal R}_{q,u}(\d) = \inf_\Delta \Big\{\prob_0(\Delta=1) + \sup_{\t\in \Theta_{q,u}(\d)}\prob_\t(\Delta=0)\Big\}
$$
where $\inf_\Delta$ is the infimum over all $\{0,1\}$-valued statistics. 
The {\it minimax rate of testing on $ \Theta_{q,u}$} is defined as $\lambda>0$, for which the following two facts hold:
\begin{itemize}
\item[(i)]  for any $\e\in (0,1)$ there exists $A_\e>0$  such that, for all $A>A_\e$,
\begin{equation}\label{test1}
{\mathcal R}_{q,u}(A\lambda) \le \e,
\end{equation}
\item[(ii)] for any $\e\in (0,1)$ there exists $a_\e>0$  such that, for all $0<A<a_\e$,
\begin{equation}\label{test2}
{\mathcal R}_{q,u}(A\lambda) \ge 1-\e.
\end{equation}
\end{itemize}
Note that this defines a non-asymptotic minimax rate of testing as opposed to the classical asymptotic definition that can be found, for example, in \cite{IngsterSuslina2003}.  A non-asymptotic minimax study of testing for the classes $B_0(s)$ and $B_q(r)$ is given by \cite{Baraud2002} and \cite{Verzelen2012}.  However,  those papers derive the minimax rates of testing on $\Theta_{q,u}$ only up to a logarithmic factor. The next theorem provides the exact expression for the minimax rates in the considered testing setup.

\begin{theorem}\label{theorem:tests}
For any integers $s$ and $d$ satisfying  $1\le s\le d$, and any $\sigma>0$, the minimax rate of testing on $ \Theta_{0,s}$ is equal to  $\lambda=(\psi_\s^{\sqrt{Q}}(s,d))^{1/2}$. For any $0< q< 2$, and any $r,\sigma>0$, the minimax rate of testing on $ \Theta_{q,r}$ is equal to  $\lambda=(\psi_{\s,q}^{\sqrt{Q}}(r,d))^{1/2}$.
\end{theorem}

The proof of this theorem consists in establishing the upper bounds \eqref{test1} and the lower bounds \eqref{test2}. We note first that the lower bounds \eqref{test2} are essentially proved in  \cite{Baraud2002} and \cite{Verzelen2012}. However, in those papers they are stated in somewhat different form, so for completeness we give a brief proof in Section \ref{sec:proof:lower}, which is very close to the proofs of the lower bounds \eqref{eq:t_l2norm2} and \eqref{eq:t_l2norm4}. The upper bounds \eqref{test1} are straightforward in view of \eqref{eq:t_l2norm1} and \eqref{eq:t_l2norm3}. Indeed, for example, to prove \eqref{test1} with $q=0$ and $u=s$, we fix some $A>0$ and consider the test
\begin{equation}\label{eq:tests1}
\Delta^* = \fcar_{\{ \hat N > (A/2) (\psi_\s^{\sqrt{Q}}(s,d))^{1/2} \}}.
\end{equation}
Then, writing for brevity $\psi=\psi_\s^{\sqrt{Q}}(s,d)$ and applying Chebyshev's inequality, we have
\begin{align}\label{eq:tests}
{\mathcal R}_{0,s}(A\psi) & \le \prob_0(\Delta^*=1) + \sup_{\t\in \Theta_{0,s}(A\sqrt{\psi})}\prob_\t(\Delta^*=0)\\
&\le \prob_0(\hat N > A \sqrt{\psi}/2) + \sup_{\t\in B_0(s)}\prob_\t(\hat N - \|\t\|_2 \le - A \sqrt{\psi}/2)\nonumber\\
&\le 2 \sup_{\t\in B_0(s)} \frac{\esp_\t(\hat N- \|\t\|_2)^2}{(A/2)^2 \psi}  \le C_*A^{-2}\nonumber
\end{align}
for some absolute constant $C_*>0$, where the last inequality follows from \eqref{eq:t_l2norm1}. Choosing $A_\e$ as a solution of $C_*A_\e^{-2}=\e$ we obtain \eqref{test1}. The case $0<q<2$ is treated analogously by introducing the test
$$
\Delta^*_q = \fcar_{\{ \hat N > (A/2) (\psi_{\s,q}^{\sqrt{Q}}(r,d))^{1/2} \}}
$$
and using \eqref{eq:t_l2norm3} rather than \eqref{eq:t_l2norm1} to get the upper bound \eqref{test1}.

Furthermore, as a simple corollary we obtain a non-asymptotic analog of the Ingster-Donoho-Jin theory.  Consider the problem of testing the hypothesis ${\bf H}_0: \t =0$ against the alternative
$
{\bf H}_1:  \  \t\in \Theta_{s}(\d)
$
where 
\begin{equation}\label{alternative}
\Theta_{s}(\d)= \{\t \in \RR^d: \ \|\t\|_0=s, \ \t_j \in\{0, \d \}, \ j=1,\dots, d\}
\end{equation}
for some integer $s\in [1,d]$ and some $\d>0$. \cite{Ingster1997} and \cite{DonohoJin2004} studied a slightly different but equivalent problem (with $\t_j$ taking values 0 and $\d$ at random) assuming in addition that $s= d^{a}$ for some $a\in(0,1/2)$. In an asymptotic setting when $\s\to 0$ and $d=d_\s\to \infty$, \cite{Ingster1997} obtained the {\it detection boundary} in the exact minimax sense, that is the value $\l = \l_\s$ such that asymptotic analogs of \eqref{test1} and \eqref{test2} hold with $A_\e=a_\e$ and $\e=0$. \cite{DonohoJin2004} proved that the detection boundary is  attained at the Higher Criticism test. Extensions to the regression and classification problems and more references can be found in \cite{IngsterPouetTsybakov2009}, \cite{IngsterTsybakovVerzelen2010}, \cite{AriasCandesPlan2011}. Note that the alternatives in these papers are defined not exactly in the same way as in  \eqref{alternative}.  

A natural non-asymptotic analog of these results consists in establishing the minimax rate of testing on $\Theta_{s}(\d)$ in the sense of the definition \eqref{test1} - \eqref{test2}. This is done in the next corollary that covers not only $\Theta_{s}(\d)$ but also
the following more general class:
$$
\Theta_{s}^*(\d)= \Big\{ \t \in \RR^d: \ \|\t\|_0=s, \  \min_{j: \ \t_j\ne 0}|\t_j| \ge \d  \Big\}.
$$
We define the minimax rate of testing on the classes $\Theta_{s}$ and $
\Theta_{s}^*$ similarly as such rate was defined for $\Theta_{q,u}$, by modifying \eqref{test1} - \eqref{test2} in an obvious way.
%%%%%%%%

\begin{corollary}\label{corollary:tests}
Let $s$ and $d$ be integers satisfying  $1\le s\le d$, and let $\sigma>0$.  The minimax rate of testing on $\Theta_{s}$ is equal $\l= \s\sqrt{\log (1+d/s^2)}$ for $s\le \sqrt{d}$. Furthermore, the minimax rate of testing on $\Theta_{s}^*$ is equal to  
$$
\lambda 
=\left\{
  \begin{array}{lcl}
 \s\sqrt{\log (1+d/s^2)} & \text{if}& s<\sqrt{d},\\ \vspace{2mm}
  \s d^{1/4}/\sqrt{s} & \text{if} & s\ge \sqrt{d}. 
   \end{array}
  \right.
$$ 
\end{corollary} 

The proof of the upper bound in this corollary is essentially the same as in Theorem~\ref{theorem:tests}. We take the same test statistic $\Delta^*$  and then act as in \eqref{eq:tests} using  that $\Theta_{s}(A\lambda)$ and $\Theta_{s}^*(A\lambda)$ are included in $\Theta_{0,s}(A\lambda\sqrt{s})$. 
The proof of the lower bound for the case $s\le \sqrt{d}$ is also the same as in Theorem~\ref{theorem:tests} since the measure $\mu_\rho$ used in the proofs (cf. Section \ref{sec:proof:lower}) is supported on $s$-sparse vectors $\t$ with all coefficients taking the same value. For $s> \sqrt{d}$ we need a slightly different lower bound argument - see Section \ref{sec:proof:lower} for the details.

\cite{Ingster1997} and \cite{DonohoJin2004} derived the asymptotic rate of testing in the form $\l=c(a)\s\sqrt{\log d}$ where the exact value $c(a)>0$ is explicitly given as a function of $a$ appearing in the relation $s=d^a$, $0<a<1/2$. Corollary~\ref{corollary:tests} allows us to explore more general behavior of $s$ leading to other types of rates. For example, we find that the minimax rate of testing is of the order $\s$ if $s=\sqrt{d}$ and it is of the order $\s \sqrt{\log\log d}$ if $s\asymp \sqrt{d}/(\log d)^\gamma$ for any $\gamma>0$. Such effects are not captured by the previous asymptotic results. Note also that the test $\Delta^*$ (cf. \eqref{eq:tests1}) that achieves the minimax rates in Corollary~\ref{corollary:tests} is very simple - it is a plug-in test based on the estimator of the $\ell_2$-norm. We do not need to invoke refined techniques as the Higher Criticism test. However, we do not prove that our method achieves the exact constant $c(a)$ in the specific regime considered by \cite{Ingster1997} and \cite{DonohoJin2004}.

%%%%%%%

\smallskip

\section{Proofs of the lower bounds}\label{sec:proof:lower}

\subsection{General tools}

Let $\mu$ be a probability measure on $\Theta$. Denote by ${\mathbb P}_\mu$ the mixture probability measure  
\begin{equation*}
{\mathbb P}_\mu = \int_\Theta \prob_\t\,\mu(d\t).
\end{equation*}
A vector $\t\in\RR^d$ is called $s$-sparse if $\|\t\|_0=s$. For an integer $s$ such that $1\le s \le d$ and $\rho>0$, we denote by $\mu_\rho$ the uniform distribution on the set of $s$-sparse vectors in $\RR^d$ with all nonzero coefficients equal to $\s\rho$. Let $$\chi^2(P',P)=\int (dP'/dP)^2 dP -1$$ 
be the chi-square divergence between two mutually absolutely continuous probability measures $P'$ and~$P$. 

The following lemma is obtained by combining arguments from \cite{Baraud2002} and~\cite{CaiLow2004}. 
\begin{lemma}\label{th_baraud} For all $\s>0, \rho>0$,  $1\le s\le d$, we have
$$
\chi^2({\mathbb P}_{\mu_\rho} ,\prob_0) \le \left(1-\frac{s}{d} + \frac{s}{d} e^{\rho^2}\right)^s -1.
$$
\end{lemma}
For completeness, the proof of this lemma is given in the Appendix. We will also need a second lemma, which is a special case of Theorem~2.15 in \cite{Tsybakov2009}:
\begin{lemma}\label{th_tsybakov}
Let $\T$ be a subset of $\RR^d$ containing 0. Assume that there exists a  probability measure $\mu$ on $\T$ and numbers $v>0, \b>0$ such that $T(\t) = 2v$ for all $\t \in {\rm supp} (\mu)$ and $\chi^2({\mathbb P}_{\mu} ,\prob_0) \le \b$, 
Then
\begin{equation*}
\inf_{\hat{T}}\sup_{\t\in\T} \prob_\t\big(|\hat{T}-T(\t)|\ge v\big) \ge \frac14\exp(-\b),
\end{equation*}
where $\inf_{\hat{T}}$ denotes the infimum over all estimators.
\end{lemma}

\subsection{Proof of the lower bound \eqref{eq:t_linear2} in~Theorem~\ref{t_linear}}
Set $\rho=\sqrt{\log (1+d/s^2)}$. Then, by Lemma~\ref{th_baraud},
\begin{equation}\label{eq:chi2}
\chi^2({\mathbb P}_{\mu_\rho} ,\prob_0) \le \left(1-\frac{s}{d} + \frac{s}{d} \left(1+\frac{d}{s^2}\right)\right)^s -1 = \left(1 +\frac{1}{s}\right)^s -1 \le e-1.
\end{equation}
Next, $L(\theta)= \s s \rho$ for all $\t \in {\rm supp} (\mu_\rho)$, and also ${\rm supp} (\mu_\rho)\subseteq B_0(s)$. Thus, the assumptions of Lemma~\ref{th_tsybakov} are satisfied with $\T=B_0(s)$, $\b=e-1$,
$v= \s s \rho/2= (1/2)\s s \sqrt{\log (1+d/s^2)}$ and $T(\t)=L(\t)$. An application of Lemma~\ref{th_tsybakov} yields
\begin{equation*}
\inf_{\hat{T}}\sup_{\t\in B_0(s)} \prob_\t\left(|\hat{T}-L(\t)|\ge (1/2)\s s \sqrt{\log (1+d/s^2)}\right) \ge \frac14\exp(1-e),
\end{equation*}
which implies  \eqref{eq:t_linear2}.

\subsection{Proof of Theorem~\ref{t_quadr_lower}}

We start by rewriting in a more convenient form the lower rates we need to prove. For this, consider separately the cases $s\ge \sqrt{d}$ and 
$s< \sqrt{d}$.

{\it Case $s\ge \sqrt{d}$.} The lower rate we need to prove in this case is $\min\{\k^4, \max(\s^2\k^2,\s^4d)\}$. It is easy to check that we can write it as follows: 
\begin{align}\label{lowerate1}
\min\{\k^4, \max(\s^2\k^2,\s^4d)\} &=
\left\{
  \begin{array}{lcl}
 \s^2\k^2 & \text{if}& \k^4 > \s^4d^2,\\
  \s^4d & \text{if} & \s^4d < \k^4 \le \s^4d^2,\\
 \k^4 &\text{if} & \k^4 \le \s^4d .
   \end{array}
  \right.
\end{align}
Note that the lower rate $ \s^4d$  for $\s^4d < \k^4 \le \s^4d^2$ follows from the lower rate $\k^4$ for $\k^4 < \s^4d$ and the fact that the minimax risk is a non-decreasing function of $\k$. Therefore, to prove Theorem~\ref{t_quadr_lower} for  $s\ge \sqrt{d}$,
it is enough to show that $R^*_Q(B_2(\k)\cap B_0(s))\ge c ( \text{lower rate})$, where $c>0$ is an absolute constant, and
\begin{equation}\label{lowerate2}
\text{lower rate}= \left\{
  \begin{array}{lcll}
 \s^2\k^2 & \text{if}& \k^4 > \s^4d^2 & \text{and} \ s= \sqrt{d},\\
 \k^4 &\text{if} & \k^4 \le \s^4d & \text{and} \ s = \sqrt{d}.
   \end{array}
  \right.
  \end{equation}
In \eqref{lowerate2}, we assume w.l.o.g. that $\sqrt{d}$ is an integer and we replace w.l.o.g. the condition $s\ge \sqrt{d}$ by $s= \sqrt{d}$ since the minimax risk is a non-decreasing function of $s$.

{\it Case $s< \sqrt{d}$.} The lower rate we need to prove in this case is $$\min\{\k^4, \max(\s^2\k^2,\s^4s^2\log^2(1+d/s^2))\}.$$   
The same argument as above shows that the analog of representation \eqref{lowerate1} holds with $d$ replaced by $s^2\log^2(1+d/s^2)$, and that
it is enough to prove the lower rate of the form:
\begin{equation}\label{lowerate3}
\text{lower rate}= \left\{
  \begin{array}{lcll}
 \s^2\k^2 & \text{if}& \k^4 > \s^4s^4\log^4(1+d/s^2) & \text{and} \ s< \sqrt{d},\\
 \k^4 &\text{if} & \k^4 \le \s^4s^2\log^2(1+d/s^2) & \text{and} \ s< \sqrt{d}.
   \end{array}
  \right.
  \end{equation}
Thus, to prove Theorem~\ref{t_quadr_lower} it remains to establish~\eqref{lowerate2} and~\eqref{lowerate3}. This is done in the following two propositions. Proposition~\ref{lowerbound1_general_case} is used with $b=\log 2$ and it is a more general fact than the first lines in~\eqref{lowerate2} and~\eqref{lowerate3} 
since $B_2(\k)\cap B_0(s)\supseteq B_2(\k)\cap B_0(1)$, and $s\log (1+d/s^2)\ge \log 2$ for $1\le s\le \sqrt{d}$. Proposition~\ref{lowerbound2_general_case} is applied with $b=1/(\log 2)$.
%%%%%%%%%%%%%%%%
\begin{proposition}\label{lowerbound1_general_case}
Let $b>0$. If $\k > b \s$, then 
\begin{equation*}
\inf_{\hat{T}}\sup_{\t\in B_2(\k)\cap B_0(1)} \prob_\t\left(|\hat{T}-Q(\t)|\ge (3b/8)\s\k  \right) \ge \frac14\exp(-b^2/4),
\end{equation*}
where $\inf_{\hat{T}}$ denotes the infimum over all estimators of $Q$.
\end{proposition}
%%%%%%%%%%%%%%%%%
\begin{proposition}\label{lowerbound2_general_case}
Let $b>0$. If $\k^4\le b^2\s^4s^2\log^2(1+d/s^2)$ and $1\le s\le d$, then 
\begin{equation*}
\inf_{\hat{T}}\sup_{\t\in B_2(\k)\cap B_0(s)} \prob_\t\left(|\hat{T}-Q(\t)|\ge \k^2/(2\max(b,1)) \right) \ge \frac14\exp(1-e),
\end{equation*}
where $\inf_{\hat{T}}$ denotes the infimum over all estimators of $Q$.
\end{proposition}

%%%%%%%%%%
\subsection{Proof of Proposition~\ref{lowerbound1_general_case}}
Consider the vectors $\t=(\k,0,\ldots,0)$ and $\t'=(\k-b\s/2,0,\ldots,0)$. Clearly, $\t$ and $\t'$ belong to $B_2(\k)\cap B_0(1)$. We have
\begin{equation*}
d(\t,\t')\triangleq \big| Q(\t)-Q(\t')\big|=|\s^2b^2/4 - \k\s b | > 3 \s\k b/4,
\end{equation*}
and the Kullback-Leibler divergence between $\prob_{\t}$ and $\prob_{\t'}$ satisfies
\begin{equation*}
K(\prob_{\t},\prob_{\t'}) = \frac{\|\t-\t'\|_2^2}{2\s^2} = \frac{b^2}{8}.
\end{equation*}
We now apply Theorem 2.2 and (2.9) in~\cite{Tsybakov2009} to obtain the result. 

%%%%%%%%%%
\subsection{Proof of Proposition~\ref{lowerbound2_general_case}} Set $\rho= \k/(\s \sqrt{\max(b,1)s}).$ Then $\rho^2 \le \log(1+d/s^2)$ and 
due to \eqref{eq:chi2} we have
$
\chi^2({\mathbb P}_{\mu_\rho} ,\prob_0) \le e-1.
$
Next, $Q(\t)= \|\t\|_2^2 = s\s^2 \rho^2=\k^2/\max(b,1)$ for all $\t \in {\rm supp} (\mu_\rho)$, which implies ${\rm supp} (\mu_\rho)\subseteq B_2(\k)$. We also have ${\rm supp} (\mu_\rho)\subseteq B_0(s)$ by construction. Therefore, the assumptions of Lemma~\ref{th_tsybakov} are satisfied with $\T=B_2(\k)\cap B_0(s)$, $\b=e-1$,
$v= \k^2/(2\max(b,1))$ and $T(\t)=Q(\t)$. An application of Lemma~\ref{th_tsybakov} yields the result.

%%%%%%%%%%%%%%%%%%%%

\subsection{Proof of Theorem~\ref{t_lin_gamma_lower}}
In order to prove Theorem~\ref{t_lin_gamma_lower}, we will need the following proposition. 
\begin{proposition}\label{lowerbound3}
Let $b>0$. If  $\k^2\le b^2\s^2s^2\log(1+d/s^2)$ and $1\le s\le d$, then 
\begin{equation*}
\inf_{\hat{T}}\sup_{\t\in B_1(\k)\cap B_0(s)} \prob_\t\left(|\hat{T}-L(\t)|\ge \k/(2\max(b,1)) \right) \ge \frac14\exp(1-e),
\end{equation*}
where $\inf_{\hat{T}}$ denotes the infimum over all estimators.
\end{proposition}
\begin{proof} We proceed as in the proof of Proposition~\ref{lowerbound2_general_case} with the following modifications. We now set 
$\rho= \k/(\max(b,1) \s s).$ Then $
\chi^2({\mathbb P}_{\mu_\rho} ,\prob_0) \le e-1
$
and $L(\t)= \|\t\|_1 = s\s \rho=\k/\max(b,1)$ for all $\t \in {\rm supp} (\mu_\rho)$, so that ${\rm supp} (\mu_\rho)\subseteq \T = B_1(\k)\cap B_0(s)$ and Lemma~\ref{th_tsybakov} applies with $\b=e-1$,
$v= \k/(2\max(b,1))$ and $T(\t)=L(\t)$. 
\end{proof}

%\subsection{}
{\it Proof of Theorem~\ref{t_lin_gamma_lower}.}
First notice that, for an integer $s\in [1,d]$, and $0<q<1$, $\k>0$,
\begin{equation}\label{eq:embed}
B_1(\k)\cap B_0(s) \subset B_q(\kq)\quad \text{if} \quad s^{1-q}\k^q = \kq^q. 
\end{equation}
We will prove the theorem by considering separately the cases $m=0$ and $m\ge 1$.

{\it Case $m=0$. }  Then, $r^2< \s^2\log (1+d)$ and the assumption of Proposition~\ref{lowerbound3} is satisfied with $s=1$,  $b=1$, and $\k=r$. Applying Proposition~\ref{lowerbound3} with these parameters and using \eqref{eq:embed} with $s=1$ we easily deduce that $R_L^*(B_q(r)) \ge Cr^2$.

{\it Case $m \ge 1$.} 
We now use the embedding \eqref{eq:embed} with $s=m$. Then
\begin{equation}\label{eq:embed1}
\k = r m^{1-1/q} \ge \s m \sqrt{\log (1+d/m^2)}
\end{equation}
where the last inequality follows from the definition of $m$. Furthermore, the fact that $m\ge 1$ and the definition of $m$ imply
\begin{equation}\label{eq:embed2}
2^{-2/q} r^2 m^{-2/q} \le r^2(m+1)^{-2/q} < \s^2 \log (1+d/(m+1)^2) \le \s^2  \log (1+d/m^2).
\end{equation}
This proves that for $\k$ defined in  \eqref{eq:embed1} we have $\k^2 \le 2^{2/q}\s^2 m^2 \log (1+d/m^2)$. Thus, the assumption of
Proposition~\ref{lowerbound3} is satisfied with $s=m$, $b=2^{1/q}$ and $\k$ defined in \eqref{eq:embed1}. 
Applying Proposition~\ref{lowerbound3} with these parameters and using \eqref{eq:embed} with $s=m$ we deduce that $R_L^*(B_q(r)) \ge C\k^2$. This and \eqref{eq:embed1} yield $R_L^*(B_q(r)) \ge C\s^2 m^2 \log (1+d/m^2)$, which is the desired lower bound.

%%%%%%%%%%%%%%%%%%%%%%%

\subsection{Proof of Theorem~\ref{t_quadr_gamma_lower}}
First notice that, for an integer $s\in [1,d]$, and $0<q<2$, $\k>0$,
\begin{equation}\label{eq:embed3}
B_2(\k)\cap B_0(s) \subset B_q(r)\quad \text{if} \quad s^{1-q/2}\k^q = \kq^q. 
\end{equation}
Consider separately the cases $m=0$, $1\le m\le \sqrt{d}$, and $m> \sqrt{d}$.

{\it Case $m=0$. }  Then, $r^2< \s^2\log (1+d)$ so that the assumption of Proposition~\ref{lowerbound2_general_case} is satisfied with $s=1$, $b=1$, and $\k=r$. Applying Proposition~\ref{lowerbound2_general_case} with these parameters and using \eqref{eq:embed3} with $s=1$ and $\k=r$ we get that  $R_Q^*(B_q(r)) \ge Cr^4$.

{\it Case $1\le m\le \sqrt{d}$. } We start by using \eqref{eq:embed3} with $s=m$. Then
\begin{equation}\label{eq:embed4}
\k = r m^{1/2-1/q} \ge \s \sqrt{m\log (1+d/m^2)}
\end{equation}
where the last inequality follows from the definition of $m$. For this $\k$, using \eqref{eq:embed2} we obtain
 $\k^2 \le 2^{2/q}\s^2 m \log (1+d/m^2)$. Thus, the assumption of
Proposition~\ref{lowerbound2_general_case} is satisfied with $s=m$, $b=2^{2/q}$ and $\k$ defined in \eqref{eq:embed4}. 
Applying Proposition~\ref{lowerbound2_general_case} with these parameters and using \eqref{eq:embed3} with $s=m$ we deduce that $R_Q^*(B_q(r)) \ge C\k^4$. This and \eqref{eq:embed4} prove the lower bound $R_Q^*(B_q(r)) \ge C\s^4 m^2 \log^2 (1+d/m^2)$.

To show that $R_Q^*(B_q(r)) \ge C\s^2r^2$, we use \eqref{eq:embed3} with $s=1$ and $\k=r$. Now, $m\ge 1$, which implies $r^2\ge \s^2\log (1+d)\ge \s^2(\log 2)$. Thus, the assumption of Proposition \ref{lowerbound1_general_case} is satisfied with $s=1$,  $\k=r$, and any $0<b<\sqrt{\log 2}$,
leading to the bound $R_Q^*(B_2(\k)\cap B_0(1)) \ge C\s^2r^2$. This inequality and the embedding in \eqref{eq:embed3} with $s=1$ yield the result.

{\it Case $m> \sqrt{d}$. }  It suffices to note that the argument used above in the case $1\le m\le \sqrt{d}$ remains valid for $m> \sqrt{d}$ and $s=\sqrt{d}$ instead of $s=m$ (assuming w.l.o.g. that $\sqrt{d}$ is an integer).

%%%%%%%%%%
\subsection{Proof of the lower bound \eqref{eq:t_l2norm2} in Theorem~\ref{t_l2norm}}\label{subsec:lower:l2norm}
Let $s<\sqrt{d}$. Set $\rho= \sqrt{\log(1+d/s^2)}$. Due to \eqref{eq:chi2} we have $
\chi^2({\mathbb P}_{\mu_\rho} ,\prob_0) \le e-1.
$
Next, $\|\t\|_2 = \s \rho \sqrt{s}=\s \sqrt{s\log(1+d/s^2)}$ for all $\t \in {\rm supp} (\mu_\rho)$, and  ${\rm supp} (\mu_\rho)\subseteq B_0(s)$ by construction. Therefore, the assumptions of Lemma~\ref{th_tsybakov} are satisfied with $\T=B_0(s)$, $\b=e-1$,
$v= \s \sqrt{s\log(1+d/s^2)}/2$ and $T(\t)=\|\t\|_2$. An application of Lemma~\ref{th_tsybakov} yields the result for $s<\sqrt{d}$. To obtain the lower bound for $s\ge \sqrt{d}$, it suffices to consider the case $s=\sqrt{d}$ (assuming w.l.o.g. that $\sqrt{d}$ is an integer) and to repeat the above argument with this value of $s$. 

%%%%%%%%%%

\subsection{Proof of the lower bound \eqref{eq:t_l2norm4} in Theorem~\ref{t_l2norm_Lqclass}}\label{subsec:lower:l2norm_Lqclass} If $m=0$ we have $r^2< \s^2\log(1+d)$. In this case, set $\rho = r/\s$, $s=1$. Then, $\rho<\sqrt{\log(1+d)}$ and due to \eqref{eq:chi2} with $s=1$ we have 
$
\chi^2({\mathbb P}_{\mu_\rho} ,\prob_0) \le 1.
$
Next, $\|\t\|_2 = \|\t\|_q= r$ for all $\t \in {\rm supp} (\mu_\rho)$. Thus, ${\rm supp} (\mu_\rho)\subseteq B_q(r)$ and the assumptions of Lemma~\ref{th_tsybakov} are satisfied with $\T=B_q(r)$, $\b=1$,
$v= r/2$ and $T(\t)=\|\t\|_2$, which implies the bound $R_{\sqrt{Q}}^*(B_q(r)) \ge Cr^2$ for $m=0$.

{\it Case $1\le m\le \sqrt{d}$. } Use the same construction as in the proof of \eqref{eq:t_l2norm2} replacing there $s$ with $m$. Then, 
$\|\t\|_2 =\s \sqrt{m\log(1+d/m^2)}$, and $\|\t\|_q =\s\rho m^{1/q} = \s m^{1/q} \sqrt{\log(1+d/m^2)}$ for all $\t \in {\rm supp} (\mu_\rho)$.
By definition of $m$, we have $\s m^{1/q} \sqrt{\log(1+d/m^2)}\le r$ guaranteeing that ${\rm supp} (\mu_\rho)\subseteq B_q(r)$. Other elements of the argument remain as in the proof of \eqref{eq:t_l2norm2}. 

{\it Case $m> \sqrt{d}$. } Use the same construction as in the proof of \eqref{eq:t_l2norm2}  with $s=\sqrt{d}$~(assuming~w.l.o.g. that $\sqrt{d}$ is an integer). Then $\rho= \sqrt{\log 2}$, $\|\t\|_2 =\s d^{1/4}\sqrt{\log 2}$, and $\|\t\|_q =\s d^{1/(2q)}\sqrt{\log 2} \le r$ (by definition of $m$)
for all $\t \in {\rm supp} (\mu_\rho)$. Other elements of the argument remain as in the proof of \eqref{eq:t_l2norm2}. 

\subsection{Proof of the lower bounds in Theorem~\ref{theorem:tests} and in Corollary~\ref{corollary:tests}} The following lemma reduces the proof to the argument, which is very close to that of the previous two proofs. 
\begin{lemma}\label{lem:tests}
If $\mu$ is a probability measure on $\Theta$, then 
$$
 \inf_\Delta \Big\{\prob_0(\Delta=1) + \sup_{\t\in \Theta}\prob_\t(\Delta=0)\Big\} \ge 1-\sqrt{\chi^2({\mathbb P}_{\mu} ,\prob_0) }
$$
where $\inf_\Delta$ is the infimum over all $\{0,1\}$-valued statistics. 
\end{lemma} 
\begin{proof} For any $\{0,1\}$-valued statistic $\Delta$,
\begin{align*}
\prob_0(\Delta=1) + \sup_{\t\in \Theta}\prob_\t(\Delta=0) &\ge \prob_0(\Delta=1) + \int_{\Theta}\prob_\t(\Delta=0) \mu( d\t)
\\
= \prob_0(\Delta=1) + {\mathbb P}_{\mu}(\Delta=0) &\ge 1- V({\mathbb P}_{\mu}, \prob_0) \ge 1- \sqrt{\chi^2({\mathbb P}_{\mu} ,\prob_0) }
\end{align*}
where $V(\cdot,\cdot)$ denotes the total variation distance and the last two inequalities follow from the standard properties of this distance (cf. Theorem~2.2(i) and (2.27) in \cite{Tsybakov2009}). 
\end{proof}

{\it Proof of the lower bound in Theorem~\ref{theorem:tests} for $q=0$.} We use a slightly modified argument of Subsection~\ref{subsec:lower:l2norm}. As in Subsection~\ref{subsec:lower:l2norm}, it suffices to prove the result in the case $s<\sqrt{d}$. Then, $\psi_\s^{\sqrt{Q}}(s,d)= \s^2 s\log(1+d/s^2)$, so that our aim is to show that the lower rate of testing on $B_0(s)$ is  $\lambda = \s \sqrt{s\log(1+d/s^2)}$. Fix $A\in(0,1)$. We use Lemma~\ref{lem:tests} with $\Theta=\Theta_{0,s}(A\l)$ and $\mu=\mu_\rho$ where 
we take $\rho=A\sqrt{\log(1+d/s^2)}$. For all $\t \in {\rm supp} (\mu_\rho)$ we have $\|\t\|_2 = \s \rho \sqrt{s}=A\l$ while  ${\rm supp} (\mu_\rho)\subseteq B_0(s)$ by construction. Hence ${\rm supp} (\mu_\rho) \subseteq \Theta_{0,s}(A\l)$, so that we can apply Lemma~\ref{lem:tests}. Next, by Lemma~\ref{th_baraud}, 
\begin{equation}\label{eqqq}
\chi^2({\mathbb P}_{\mu_\rho} ,\prob_0) \le \left(1-\frac{s}{d} + \frac{s}{d} \left(1+\frac{d}{s^2}\right)^{A^{2}}\right)^s -1 \le \left(1 +\frac{A^2}{s}\right)^s -1 \le \exp(A^2)-1
\end{equation}
where we have used that $(1+x)^{A^{2}}-1\le A^2 x$ for $0<A<1$, $x>0$. The last display and Lemma~\ref{lem:tests} imply that ${\mathcal R}_{0,s}(A\lambda)\ge 1-\sqrt{\exp(A^2)-1}$. Choosing $a_\e$ such that $\sqrt{\exp(a_\e^2)-1}=\e$ proves \eqref{test2}.

{\it Proof of the lower bound in Theorem~\ref{theorem:tests} for $0<q<2$} follows along similar lines but now we modify, in the same spirit, the argument of Subsection~\ref{subsec:lower:l2norm_Lqclass} rather than that of Subsection~\ref{subsec:lower:l2norm}. The corresponding  $\rho$ in Subsection~\ref{subsec:lower:l2norm_Lqclass} is multiplied by a suitable $A\in(0,1)$ and then Lemma~\ref{lem:tests} is applied. We omit the details.

{\it Proof of the lower bound in Corollary~\ref{corollary:tests}.} As explained after the statement of Corollary~\ref{corollary:tests}, we need only to consider the case $s>\sqrt{d}$ for the class $\Theta_{s}^*$. Then, $\l=\s d^{1/4}/\sqrt{s}$. Instead of $\mu_\rho$ we consider now a slightly different measure $\bar\mu_\rho$, which is the uniform distribution on the set of $s$-sparse vectors in $\RR^d$ with nonzero coefficients taking values in $\{-\s\rho, \s\rho\}$. 
Then, similarly to Lemma~\ref{th_baraud}, 
\begin{equation}\label{eqqq1}
\chi^2({\mathbb P}_{\bar\mu_\rho} ,\prob_0) \le \Big(1-\frac{s}{d} + \frac{s}{d} \cosh(\rho^2)\Big)^s -1,
\end{equation}
cf. formula (27) in \cite{Baraud2002}. Fix $A\in(0,1)$. We now use Lemma~\ref{lem:tests} with $\Theta=\Theta_{s}^*(A\l)$ and $\mu=\bar\mu_\rho$ where 
we take $\rho=Ad^{1/4}/\sqrt{s}$. For all $\t \in {\rm supp} (\bar\mu_\rho)$ we have $|\t_j|  = \s \rho =A\s d^{1/4}/\sqrt{s} = A\l$ and also ${\rm supp} (\bar\mu_\rho)\subseteq \{\|\t\|_0=s\}$ by construction. Hence ${\rm supp} (\bar\mu_\rho) \subseteq \Theta_{s}^*(A\l)$, so that we can apply Lemma~\ref{lem:tests}. Since $s>\sqrt{d}$ we have $\rho<1$. Using \eqref{eqqq1} and the fact that $\cosh(x)\le 1+x^2$ for $0<x<1$ we obtain
$$
\chi^2({\mathbb P}_{\bar\mu_\rho} ,\prob_0) \le \Big(1 + \frac{s\rho^4}{d} \Big)^s -1  \le \exp(A^4)-1
$$
and we conclude the proof in the same way as it is done after \eqref{eqqq}. 

%%%%%%%%%%%%%%%%%

\smallskip

\section{Proofs of the upper bounds}\label{sec:proof:upper}

We will use the following lemma.

%%%%%%%
\begin{lemma}\label{alpha}
For $X\sim\nzeroun$ and any $x>0$ we have
\begin{equation}\label{eq:alpha1}
\frac{4}{\sqrt{2\pi}(x+\sqrt{x^2+4})}e^{-x^2/2} \le \prob\big(|X|>x\big) \le \frac{4}{\sqrt{2\pi}(x+\sqrt{x^2+2})}e^{-x^2/2}, 
\end{equation}
\begin{equation}\label{eq:alpha2}
\esp \Big[X^2\fcar_{ \{|X|>x\}}\Big] \le \sqrt{2 \over \pi}\Big( x+{2 \over x} \Big) e^{-x^2/2}, 
\end{equation}
\begin{equation}\label{eq:alpha3}
\esp \Big[X^4\fcar_{ \{|X|>x\}}\Big] \le \sqrt{2 \over \pi}\Big( x^3+ 3x + \frac1x \Big) e^{-x^2/2}.
\end{equation}
\end{lemma}
%%%%%%
Inequality \eqref{eq:alpha1} is due to \cite{Birnbaum1942} and \cite{Sampford1953}. Inequalities  \eqref{eq:alpha2} and  \eqref{eq:alpha3} follow from integration by parts.

\smallskip

In this section, we will use the notation
$$x=\sqrt{2\log (1+d/s^2)}, \quad \hat{S}=\{j:  |y_j|>\s x\}, \quad S=\{j:  \t_j \ne 0\}.$$
We will denote by $C_i, i=1,2,\dots$, absolute positive constants, and by $C$ absolute  positive constants that can vary from line to line. 

\subsection{Proof of the bound \eqref{eq:t_linear1} in Theorem~\ref{t_linear}}

Clearly,  $ \esp_\t(\sum_{j=1}^d y_j -L(\t))^2 = \s^2d$. Thus,  in view of \eqref{eq:t_linear3}, to prove  \eqref{eq:t_linear1} it is enough to show that for $s\le \sqrt{d}$ we have 
\begin{equation}\label{eq:proof:t_linear1_1}
\sup_{\t\in B_0(s)}\esp_\t(\hat L_* -L(\t))^2 \le C \s^2 s^2 \log (1+d/s^2)
\end{equation}
where 
$$\hat L_* = \sum_{j=1}^d y_j ~\fcar_{\{ |y_j|>\s\sqrt{2\log (1+d/s^2)} \}}$$
and $C>0$ is an absolute constant. We have
\begin{equation}\label{eq:proof:t_linear1_2}
\hat L_* -L(\t) = \sum_{j\in S} (y_j- \t_j ) - \sum_{j\in S\backslash\hat{S}} y_j + \sum_{j\in \hat{S}\backslash S}  y_j .
\end{equation}
Thus, for $\t\in B_0(s)$, we obtain
\begin{eqnarray*}
\esp_\t(\hat L_* -L(\t))^2
 &\le & 3~\esp\Big(\sum_{j\in S} \s\xi_j\Big)^2 +3~\esp_\t\Big(\sum_{j\in S} y_j ~\fcar_{\{ |y_j|\le \s x \}}\Big)^2 +3~\esp\Big( \sum_{j\in S^c} \s \xi_j ~\fcar_{\{ |\xi_j|> x \}} \Big)^2\\
  &\le & 3 \s^2 \Big\{ (s + s^2x^2)  +\sum_{j\in S^c} \esp\Big(  \xi_j^2 ~\fcar_{\{ |\xi_j|> x \}} \Big)\Big\} \phantom{\sqrt{2 \over \pi}}
  \\
  &\le & 3 \s^2 \Big\{ (s + s^2x^2) + d \sqrt{2 \over \pi}\Big( x+{2 \over x} \Big) e^{-x^2/2}\Big\} \quad \text{(by \eqref{eq:alpha2})} \phantom{\sum_{j\in S^c}}
  \\
  &\le & 3 \s^2 \Big\{ (s + s^2x^2) + s^2 \sqrt{2 \over \pi}\Big( x+{2 \over x} \Big)\Big\}, \phantom{\sum_{j\in S^c}}
\end{eqnarray*}
and \eqref{eq:proof:t_linear1_1} follows since $x\ge \sqrt{2\log 2}$ for  $s\le \sqrt{d}$.

\subsection{Proof of Theorem~\ref{t_lin_gamma_upper}}

We will consider only the sparse zone $1\le m\le \sqrt{d}$ since the cases $m=0$ and $m>\sqrt{d}$ are trivial. Fix $\t\in B_q(r)$. We will use the notation 
$$
\tilde{d} = 1+d/m^2, \quad  \tilde{x}=2\sqrt{2\log\tilde{d}}, \quad \ts=\{j:  |\t_j|> \s\tilde{x}/2\}.
$$
Note that  
\begin{equation}\label{eq:proof:t3_1}
\card( \ts) \le \left(\frac{2\kq}{\s\tilde{x}}\right)^q < 2^{-q/2} (m+1) \le 2^{1-q/2} m,
\end{equation}
where the first inequality is due to the fact that $\t\in B_q(r)$ and the second follows from the definition of $m$.

Consider first the bias of $\hat L_q$. 
Lemma~\ref{biais} yields
\begin{align}\label{eq:proof:t3_2}
\big(\esp_\t (\hat{L}_q) - L(\t)\big)^2&\leq C \Big(\sum_{ j=1}^d \min (|\t_j|,\s \tilde{x})\Big)^2\le C \Big(\sum_{ j=1}^d |\t_j|^q(\s \tilde{x})^{1-q}\Big)^2\\
&\le C\left(\frac{\kq}{\s\tilde{x}}\right)^{2q}\s^{2}\log \tilde{d}\nonumber\\
&\le C \s^2 m^2 \log \tilde{d},\nonumber
\end{align}
where we have used \eqref{eq:proof:t3_1}.
Next, the variance of $\hat L_q$ has the form
$$
\var_\t (\hat L_q)=\sumjd \var_\t(y_j~\fcar_{\{ |y_j|>\s\tilde{x} \}}).
$$
Here, for indices $j$ belonging to $\ts$, using \eqref{eq:proof:t3_1} we have
\begin{align}\label{eq:proof:t3_3}
\sum_{j \in \ts} \var_\t(y_j~\fcar_{\{ |y_j|>\s\tilde{x} \}})&\le2\sum_{j \in \ts} \var_\t(y_j)+ 2 \sum_{j \in \ts} \var_\t(y_j~\fcar_{\{ |y_j|\le \s\tilde{x} \}})\\
&\le 2 \card(\ts)\s^2(1+\tilde{x}^2)\nonumber\\
&\le C \s^2 m \log \tilde{d}.\nonumber
\end{align}
For indices $j$ belonging to $\ts^c$, we have
\begin{align}\label{eq:proof:t3_4}
\sum_{j \in \ts^c} \var_\t(y_j~\fcar_{\{ |y_j|>\s\tilde{x} \}})&\le \sum_{j \in \ts^c} \esp_\t(y_j^2~\fcar_{\{ |y_j|>\s\tilde{x} \}})\\
&\le 2\sum_{j \in \ts^c} \t_j^2+2\s^2\sum_{j \in \ts^c}\esp_\t (\xi_j^2~\fcar_{\{ |y_j|>\s\tilde{x} \}})\nonumber \\
&\le 2\Big(\sum_{j \in \ts^c} |\t_j|\Big)^2+2\s^2\sum_{j\in \ts^c}   \esp (\xi_j^2~\fcar_{\{ |\xi_j|>\sqrt{2\log \tilde{d}} \}}).\nonumber
\end{align}
Using the same argument as in \eqref{eq:proof:t3_2} we find
\begin{equation}\label{eq:proof:t3_5}
\Big(\sum_{j \in \ts^c} |\t_j|\Big)^2 \le C\Big(\sum_{j=1}^d \min (|\t_j|,\s \tilde{x})\Big)^2\le C \s^2 m^2 \log \tilde{d}.
\end{equation} 
Finally,  \eqref{eq:alpha2} implies 
\begin{align}\label{eq:proof:t3_6}
\s^2\sum_{j\in \ts^c}\esp (\xi_j^2~\fcar_{\{ |\xi_j|>\sqrt{2\log \tilde{d}} \}})
&\le C \s^2(d/\tilde{d})\sqrt{\log\tilde{d}} \le C \s^2 m^2 \log\tilde{d}
\end{align}
where for the last inequality we have used that $\log\tilde{d} \ge \log 2$ for $m\le \sqrt{d}$. Combining \eqref{eq:proof:t3_3} -- \eqref{eq:proof:t3_6} we obtain that 
$$
\var_\t (\hat L_q)\le C \s^2 m^2 \log\tilde{d}.
$$
Together with \eqref{eq:alpha2}, this yields the desired result:
$$\sup_{\t\in  B_q(\kq)} \esp_\t(\hat{L}_q - L(\theta))^2 \le C \s^2 m^2 \log\tilde{d}. $$

\subsection{Proof of Theorem~\ref{t_quadr_upper}}

The upper bound $\k^4$ for $\k^4 < \psi_\s(s,d,\k)$ is trivial since the risk of the zero estimator is equal to $\k^4$.  Let now $\k^4 \ge  \psi_\s(s,d,\k)$. We analyze separately the cases $s\ge \sqrt{d}$, $\k^4 \ge  \psi_\s(s,d,\k)$, and $s< \sqrt{d}$, $\k^4 \ge  \psi_\s(s,d,\k)$. 

{\it Case $s\ge \sqrt{d}$ and $\k^4 \ge  \psi_\s(s,d,\k)$.} Then, $\hat Q= \hat Q_*$ and Theorem~\ref{t_quadr_upper} claims a bound~with the rate $\psi_\s^Q(s,d,\k) = \psi_\s(s,d,\k)=\max(\s^2\k^2,\s^4d)$. To prove this bound, note that 
\begin{equation*}
\hat Q_* - Q(\t) = 2\s \sumjd \t_j\xi_j + \s^2 \sumjd (\xi_j^2-1).
\end{equation*}
Thus, for all $\t\in B_2(\k)$,
\begin{align}\nonumber
\esp_\t(\hat Q_* - Q(\t) )^2 &= 4\s^2 \esp\Big(\sumjd \t_j\xi_j\Big)^2 + \s^4 \esp\Big(\sumjd (\xi_j^2-1)\Big)^2 \\
&= 4\s^2 \|\t\|_2^2 + 2\s^4d \le 6 \max(\s^2\k^2,\s^4d).\label{proof:1}
\end{align}

{\it Case $s< \sqrt{d}$ and $\k^4 \ge  \psi_\s(s,d,\k)$.} Then, $\hat Q= \hat Q'$ where
$$
\hat Q' = \sum_{j=1}^d (y_j^2 -\a \s^2)~\fcar_{\{ |y_j|>\s\sqrt{2\log (1+d/s^2)} \}}
$$
and  $\psi_\s^Q(s,d,\k) =\max(\s^2\k^2,\s^4s^2\log^2(1+d/s^2))$. Here and below in this proof, we set for brevity $\a=\a_s$. 

Since $s< \sqrt{d}$, we have $x\ge \sqrt{2\log 2}$. Using Lemma \ref{alpha}, we find that, for  $s\le \sqrt{d}$, 
\begin{align}\label{eq:alpha}
\a &= \frac{\esp \Big(X^2\fcar_{ \{|X|>x\}}\Big)}{\prob\big(|X|>x\big)} \le (x+2/x)(x+1) \le 5x^2= 10\,  \log (1+d/s^2).
\end{align}
Similarly to \eqref{eq:proof:t_linear1_2}, we get
\begin{equation*}
\hat{Q}' - Q(\t) = \sum_{j\in S} ( y_j^2 - \a\s^2 - \t_j^2) - \sum_{j\in S\backslash\hat{S}} (y_j^2-\a\s^2) + \sum_{j\in \hat{S}\backslash S} (y_j^2 - \a\s^2),
\end{equation*}
and thus
\begin{equation*}
\esp_\t\big(\hat{Q}' - Q(\t)\big)^2 \le 3~\esp_\t\Big[\Big(\sum_{j\in S} ( y_j^2 -\a\s^2- \t_j^2)\Big)^2 +\Big(\sum_{j\in S\backslash\hat{S}} (y_j^2-\a\s^2)\Big)^2 +\Big( \sum_{j\in \hat{S}\backslash S} ( y_j^2 - \a\s^2)\Big)^2\Big].
\end{equation*}
For $\t \in B_2(\k)\cap B_0(s)$, the first term on the right-hand side satisfies
\begin{align}
\esp_\t \Big( \sum_{j\in S} ( y_j^2 -\a\s^2 - \t_j^2) \Big)^2 &= \esp\Big( \sum_{j\in S} ( 2\s\t_j\xi_j  + \s^2(\xi_j^2-\a))\Big)^2 \nonumber\\
&\le  4\s^2\|\t\|_2^2+ 2\s^4s^2(\a^2 + 3) \le 4\s^2\|\t\|_2^2+ 2\s^4s^2(25x^4 + 3)\nonumber\\
&\le C_1\ \big( \s^2\|\t\|_2^2 + \s^4s^2\log^2(1+d/s^2) \big) \label{proof:2}\\
&\le C_1\ \big( \s^2\k^2 + \s^4s^2\log^2(1+d/s^2) \big).\nonumber
\end{align}
Furthermore, by definition of $\hat{S}$,
\begin{align*}
\esp_\t \Big( \sum_{j\in S\backslash\hat{S}} ( y_j^2 -\a\s^2 )~\Big)^2 &\le 4\s^4s^2\log^2(1+d/s^2)+2\s^4s^2\a^2 \\
&\le C_2\s^4s^2\log^2(1+d/s^2)
\end{align*}
 for any $\t \in B_0(s)$. 
Finally, $\a$ was chosen such that, for any $j\not \in S$,
\begin{equation*}
\esp_\t \bigg[ \big( y_j^2-\a\s^2 \big)\fcar_{\{ |y_j|> \s x\}}\bigg]=\s^2\esp \bigg[\big( X^2-\a \big) \fcar_{\{ |X|> x\}}\bigg]=0,
\end{equation*}
where $X\sim\nzeroun$.
Thus, by independence we have 
\begin{align}
\esp_\t \Big( \sum_{j\in \hat{S}\backslash S} ( y_j^2 - \a\s^2) \Big)^2 &= \sum_{j\not\in S} \esp_\t \bigg[ \big( y_j^2-\a\s^2 \big)^2\fcar_{\{ |y_j|> \s x\}}\bigg] \nonumber\\
\label{vvv}
&\le \s^4 d\, \esp \bigg[\big( X^2-\a \big)^2 \fcar_{\{ |X|> x\}}\bigg] \phantom{\sum_{j\not\in S} \esp_\t \bigg[}\\
&\le 16 \s^4 d\, \esp \Big[X^4\fcar_{\{ |X|> x\}}\Big] \le C_3 \s^4 d \,x^3 e^{-x^2/2} \phantom{\sum_{j\not\in S} \esp_\t \bigg[}\nonumber \\
&\le 
 C_4 \s^4 s^2 x^3 \le (C_4/\sqrt{2\log 2}) \s^4 s^2 x^4 \le C_5 \s^4 s^2\log^2(1+d/s^2), \nonumber
  \phantom{\sum_{j\not\in S} \esp_\t \bigg[}
\end{align}
where  we have used that
$\a \le 5X^2$ on the event $\{|X|>x\}$,  inequality \eqref{eq:alpha3} and the fact that $x\ge \sqrt{2\log 2}$.
Combining the above displays yields  
$$
\sup_{\t \in B_2(\k)\cap B_0(s)}\esp_\t\big(\hat{Q}' - Q(\t)\big)^2 \le C_6\max(\s^2\k^2,\s^4s^2\log^2(1+d/s^2)).%= C_6 \psi_\s^Q(s,d,\k).
$$
%%%%%%%%%%%%%%%

\subsection{Proof of Theorem~\ref{t_quadr_gamma_upper}}

Fix  $\t\in B_q(r)$. 
We will prove the theorem only for $1\le m\le \sqrt{d}$ since the case $m=0$ is trivial and the result for the case $m>\sqrt{d}$ follows from \eqref{proof:1} and the fact that $\|\t\|_2\le \|\t\|_q\le r$. In this proof, we will write for brevity $\a= \tilde \a_m$, $\tilde d = 1+d/m^2$, $\tilde x = 2(2\log \tilde d)^{1/2}$. Let $J\subseteq \{1,\dots,d\}$ be the set of indices corresponding to the $m$ largest in absolute value components of $\theta$, and let $|\theta|_{(j)}$ denote the $j$th largest absolute value of the components of $\theta$. It is easy to see that
\begin{equation*}
|\theta|_{(j)}\le \frac{\|\theta\|_q}{j^{1/q}}.
 \end{equation*}
This implies
\begin{equation*} 
\sum_{j\in J^c} \t_j^2= \sum_{j\ge m+1}|\theta|_{(j)}^2\le |\theta|_{(m)}^{2-q} \sum_{j\ge m+1} |\theta|_{(j)}^{q}\le  \left(\frac{\|\theta\|_q}{m^{1/q}}\right)^{2-q}\|\theta\|_q^q =  \|\theta\|_q^2 m^{1-2/q}.
\end{equation*}
Therefore, since $\t\in B_q(r)$ and due to the definition of $m$,
\begin{equation}\label{eq:t_quadr_uper1} 
\sum_{j\in J^c} \t_j^2 \le r^2 m^{1-2/q} \le \s^2m \log \tilde d,
\end{equation}
and
\begin{equation}\label{eq:t_quadr_uper2} 
\forall \ j\in J^c :  \quad |\theta_j|\le r m^{-1/q} \le \s \sqrt{\log \tilde d} \le \s \tilde x /2.
\end{equation}

%%%%%%%%%%%%%%%%%
We have
\begin{align}\label{eq:t_quadr_uper3} 
\hat{Q}_q - Q(\t) &= \sum_{j\in J} \big\{ y_j^2 - \a\s^2 - \t_j^2\big\} - \sum_{j\in J\backslash\ts} \big\{ y_j^2  - \a \s^2\big\} 
\\ 
&+ \sum_{j\in \ts\backslash J} \big\{ y_j^2 - \a\s^2\big\} - \sum_{j\in J^c} \t_j^2 .\nonumber
\end{align}
Consider the first sum on the right hand side of \eqref{eq:t_quadr_uper3}. Since ${\rm Card}(J)=m$, and $\a\le 40 \log \tilde d$ (which is obtained analogously to \eqref{eq:alpha} recalling that now $\a=\tilde\a_m$ instead of $\a=\a_s$), the same argument as in \eqref{proof:2}
leads to
\begin{align}\label{eq:t_quadr_uper4} 
\esp_\t \Big( \sum_{j\in J} \big\{ y_j^2 - \a\s^2 - \t_j^2\big\} \Big)^2 &\le C ( \s^2\|\t\|_2^2+ \s^4 m^2 \log^2 \tilde d).  
\end{align}
Next, consider the second sum on the right hand side of \eqref{eq:t_quadr_uper3}. By definition of $\ts$,
\begin{align}\label{eq:t_quadr_uper5} 
\esp_\t \Big( \sum_{j\in J\backslash\ts} \big\{ y_j^2 - \a\s^2\big\} ~\Big)^2 &\le
\Big( \sum_{j\in J} \s^2 (\tilde x +\a) \Big)^2\le C \s^4 m^2 \log^2\tilde{d}. 
\end{align}
Let us now turn to the third sum on the right hand side of \eqref{eq:t_quadr_uper3}. The bias-variance decomposition yields
\begin{align*}
&\esp_\t \Big(\sum_{j\in \ts\backslash J} \big\{ y_j^2 - \a\s^2\big\} ~\Big)^2 = 
\esp_\t \Big(\sum_{j\in J^c} ( y_j^2 - \a\s^2)~\fcar_{\{ |y_j|> \s \tilde x\}}\Big)^2\\
&=\sum_{j\in J^c} {\rm Var}_\t \Big(( y_j^2 - \a\s^2)~\fcar_{\{ |y_j|> \s \tilde x\}}\Big) +
 \Big[\sum_{j\in J^c} \esp_\t\Big(( y_j^2 - \a\s^2)~\fcar_{\{ |y_j|> \s \tilde x\}}\Big)\Big]^2.
\end{align*}
Here,
\begin{align*}
{\rm Var}_\t \Big(( y_j^2 - \a\s^2)~\fcar_{\{ |y_j|> \s \tilde x\}}\Big) &\le \esp_\t \Big(( y_j^2 - \a\s^2)~\fcar_{\{ |y_j|> \s \tilde x\}}\Big)^2\\
&\le C\esp_\t \Big((\t_j^4 + \s^4\xi_j^4 + \a^2\s^4)~\fcar_{\{ |y_j|> \s \tilde x\}}\Big)\\
&\le C\Big[\t_j^4 + \a^2\s^4 + \s^4 \esp \Big(\xi_j^4~\fcar_{\{ |\xi_j|> \tilde x/2\}}\Big)\Big] \quad \text{(by \eqref{eq:t_quadr_uper2})}.
\end{align*}
Using now the same argument as in \eqref{vvv} to bound $\esp \big(\xi_j^4~\fcar_{\{ |\xi_j|> \tilde x/2\}}\big)$ we obtain 
\begin{align*}
\sum_{j\in J^c} {\rm Var}_\t \Big(( y_j^2 - \a\s^2)~\fcar_{\{ |y_j|> \s \tilde x\}}\Big)&\le C\big(\sum_{j\in J^c}\t_j^4 + \s^4 m^2 \log^2 \tilde{d}\big)
\nonumber\\
&\le C\Big(\Big(\sum_{j\in J^c}\t_j^2\Big)^2 + \s^4 m^2 \log^2 \tilde{d}\Big).
\end{align*}
Furthermore, by Lemma~\ref{biais2}, 
$$
\Big|\sum_{j\in J^c} \esp_\t\Big(( y_j^2 - \a\s^2)~\fcar_{\{ |y_j|> \s \tilde x\}}\Big)\Big|\le C \sum_{j\in J^c}\t_j^2.
$$
Combining the above displays leads to the following bound :
\begin{align}\label{eq:t_quadr_uper6} 
&\esp_\t \Big(\sum_{j\in \ts\backslash J} \big\{ y_j^2 - \a\s^2\big\} ~\Big)^2 \le
C\Big(\Big(\sum_{j\in J^c}\t_j^2\Big)^2 + \s^4 m^2 \log^2 \tilde{d}\Big).
\end{align}
From \eqref{eq:t_quadr_uper3} - \eqref{eq:t_quadr_uper6} we deduce that
$$
\esp_\t(\hat{Q}_q-Q(\t))^2 \le
C\Big(\s^2\|\t\|_2^2+ \Big(\sum_{j\in J^c}\t_j^2\Big)^2 + \s^4 m^2 \log^2 \tilde{d}\Big).
$$
The result now follows if we use \eqref{eq:t_quadr_uper1} and note that $\|\theta\|_2\le \|\theta\|_q\le r$.

\subsection{Proof of the upper bound \eqref{eq:t_l2norm1} in Theorem~\ref{t_l2norm}} Fix $\t \in B_0(s)$ and set for brevity $\tau = \big(\psi_\s^{\sqrt{Q}}(s,d)\big)^{1/2}$. We will bound the risk $\esp_\t(\hat{N} - \|\t\|_2)^2$ separately for the cases $\|\t\|_2\le \tau$ and  $\|\t\|_2> \tau$. 

\smallskip

{\it Case $\|\t\|_2\le \tau$.} Using the elementary inequality $(a-b)^2 \le 2(a^2-b^2) + 4b^2$, we find
\begin{equation*}
\esp_\t(\hat{N} - \|\t\|_2)^2 \le  2~ \esp_\t(\max\{\hat{Q}_{\bullet},0\} - Q(\t)) + 4 Q(\t) \le 2 \left(\esp_\t(\hat{Q}_{\bullet} - Q(\t))^2\right)^{1/2} + 4\tau^2.
\end{equation*}
Note that $\hat{Q}_{\bullet} = \hat Q$ if we set $\k=\tau$ in the definition of $\hat Q$. Furthermore, $\t \in B_0(s)$ and, in the case under consideration $\t$ belongs to $B_2(\tau)$. Now, use that for all  $\t \in B_2(\tau)\cap B_0(s)$, due to Theorem~\ref{t_quadr_upper}, we have
$$
\esp_\t(\hat{Q}_{\bullet} - Q(\t))^2 \le C\psi_\s^{Q}(s,d, \tau). 
$$
Using this inequality and the fact that $\psi_\s^{Q}(s,d, \tau)=\big(\psi_\s^{\sqrt{Q}}(s,d)\big)^2$, we obtain the desired rate:
\begin{equation*}
\esp_\t(\hat{N} - \|\t\|_2)^2 \le   C_7 \psi_\s^{\sqrt{Q}}(s,d) + 4\tau^2= (C_7 +4) \psi_\s^{\sqrt{Q}}(s,d).
\end{equation*}

\smallskip

{\it Case $\|\t\|_2 > \tau$.} Using the elementary inequality $\forall \ a>0, b\ge 0,$ $(a-b)^2 \le (a^2-b^2)^2/a^2$,  we find
\begin{equation*}
\esp_\t(\hat{N} - \|\t\|_2)^2 \le  \frac{\esp_\t(\hat{Q}_{\bullet} - Q(\t))^2}{\|\t\|_2^2}.
\end{equation*}
Now, we bound $\esp_\t(\hat{Q}_{\bullet} - Q(\t))^2$ along the lines of the proof of Theorem~\ref{t_quadr_upper}. In particular, if $s\ge \sqrt{d}$ we have $\hat{Q}_{\bullet}= \hat Q_*$, $\tau = \s d^{1/4}$ and using \eqref{proof:1} we obtain
$$
\frac{\esp_\t(\hat{Q}_{\bullet} - Q(\t))^2}{\|\t\|_2^2}  \le 4\s^2  + \frac{2\s^4d}{\|\t\|_2^2} \le 4\s^2  + \frac{2\s^4d}{\tau^2} \le C_8 \s^2 \sqrt{d},
$$
which is the desired rate. If $s< \sqrt{d}$, we have $\hat{Q}_{\bullet}= \hat Q'$, $\tau = \s \sqrt{s\log(1+d/s^2)}$ and using \eqref{proof:2} and the subsequent bounds in the proof of Theorem~\ref{t_quadr_upper}, we obtain
\begin{align}\label{eq:bullet}
\frac{\esp_\t(\hat{Q}_{\bullet} - Q(\t))^2}{\|\t\|_2^2} & \le 
 \frac{3\big( C_1\s^2\|\t\|_2^2 + (C_1+C_2+C_5)\s^4s^2\log^2(1+d/s^2) \big)}{\|\t\|_2^2} \\
 & \le C_9\left(\s^2 + \frac{\s^4s^2\log^2(1+d/s^2)}{\tau^2}\right) \le  C_{10} \s^2s \log (1+d/s^2), \nonumber
 \end{align}
which is again the desired rate.
%%%%%%%%%%%%%%%

\subsection{Proof of the upper bound \eqref{eq:t_l2norm3} in Theorem~\ref{t_l2norm_Lqclass}} 

The case $m=0$ is trivial. For $m\ge 1$, we use the same method of reduction to the risk of estimators of $Q$ as in the proof of~\eqref{eq:t_l2norm1}.  The difference is that now we set
$\tau = \big(\psi_{\s,q}^{\sqrt{Q}}(r,d)\big)^{1/2}$, we replace $s$ by $m$, and we apply Theorem~\ref{t_quadr_gamma_upper} rather than to Theorem~\ref{t_quadr_upper}. In particular, an analog of \eqref{eq:bullet} with $s=m$ is obtained using \eqref{eq:t_quadr_uper4}.

%%%%%%%%%%%%%%%%%%%
\subsection{Proof of Theorem~\ref{sigma_lin}}

As in the proof of the bound \eqref{eq:t_linear1} and with the same notation, we have,  for $\t\in B_0(s)$, 
\begin{eqnarray*}
\esp_\t(\tilde L -L(\t))^2
 &\le & 3\esp\Big(\sum_{j\in S} \s\xi_j\Big)^2 +3\esp_\t\Big(\sum_{j\in S} y_j ~\fcar_{\{ |y_j|\le \hsig x \}}\Big)^2 +3\esp\Big( \sum_{j\in S^c} \s \xi_j ~\fcar_{\{ \s|\xi_j|> \hsig x \}} \Big)^2\\
 &\le & 3 \Big\{ (s\s^2 + s^2\esp_\t(\hsig^2)x^2)  +\s^2\sum_{j\in S^c} \esp\Big(  \xi_j^2 ~\fcar_{\{ \s|\xi_j|> \hsig x \}} \Big)\Big\}.
  \end{eqnarray*}
%In this proof, $C$ is an absolute  positive constant that can vary from line to line. 
Here,
 \begin{multline*}
 $$\esp_\t\Big(\xi_j^2~\fcar_{\{ \s|\xi_j|>\hsig\sqrt{2\log (1+d/s^2)} \}}\Big)\\=\esp_\t\Big(\xi_j^2~\fcar_{\{ \s|\xi_j|>\hsig\sqrt{2\log (1+d/s^2)} \}}\fcar_{\{\hsig>\s\}}\Big)+ \esp_\t\Big(\xi_j^2~\fcar_{\{ \s|\xi_j|>\hsig\sqrt{2\log (1+d/s^2)} \}}\fcar_{\{\hsig\leq\s\}}\Big).
 \end{multline*}
The first term on the right hand side satisfies
\begin{align*}
\esp_\t\Big(\xi_j^2~\fcar_{\{ \s|\xi_j|>\hsig\sqrt{2\log (1+d/s^2)} \}}\fcar_{\{\hsig>\s\}}\Big)&\leq 
\esp_\t\Big(\xi_j^2~\fcar_{\{ |\xi_j|>\sqrt{2\log (1+d/s^2)}\}}\Big)\\
&\leq \frac{Cs^2}{d}\sqrt{\log (1+d/s^2)} \quad \text{(by \eqref{eq:alpha2})}.
\end{align*}
For the second term, we use Lemma \ref{sig} to get
$$\esp_\t\Big(\xi_j^2~\fcar_{\{ \s|\xi_j|>\hsig\sqrt{2\log (1+d/s^2)} \}}\fcar_{\{\hsig\leq\s\}}\Big)\leq \sqrt{\esp(\xi_1^4)}\sqrt{\prob_\t(\hsig\leq \s)}
\le C\sqrt{d} \exp(-\sqrt{d}/C).$$
Combining the above displays and using Lemma \ref{sig} to bound $\esp_\t(\hsig^2)$ we obtain
  $$\esp_\t(\tilde L -L(\t))^2 \le C \s^2 s^2 \log (1+d/s^2).$$  
%%%%%%%%%%%%%%%%%%%%%%%%%%%%%%%%

\subsection{Proof of Theorem~\ref{theorem_upper_bound_unknown_noise}}
Set ${\tilde S} = \{j: |y_j|\ge \hat{\s} \sqrt{2\log d} \}$.
Similarly as in the proof of~Theorem~\ref{t_quadr_upper}, we get 
\begin{equation*}
\esp_\t\big(\tilde{Q} - Q(\t)\big)^2 \le 3~\esp_\t\Big[\Big(\sum_{j\in S} ( y_j^2 - \t_j^2)\Big)^2 +\Big(\sum_{j\in S\backslash\tilde{S}} y_j^2\Big)^2 +\Big( \sum_{j\in \tilde{S}\backslash S} y_j^2 \Big)^2\Big].
\end{equation*}
We bound separately the three terms on the right hand side. 
For $\t \in B_2(\k)\cap B_0(s)$, the first term on the right-hand side satisfies, due to \eqref{proof:2} with $\a=0$,
\begin{align}\label{proof:1n}
\esp_\t \Big( \sum_{j\in S} ( y_j^2 - \t_j^2) \Big)^2
&\le C\ \big( \s^2\|\t\|_2^2 + \s^4s^2\big) %\quad \text{(by Lemma~\ref{sig})}
\le C\ \big( \s^2\k^2 + \s^4s^2 \big).%\nonumber
\end{align}
Using Lemma~\ref{sig} we find  
\begin{align}
\esp_\t  \Big(\sum_{j\in S\backslash\tilde{S}} y_j^2\Big)^2& = \esp_\t  \Big(\sum_{j\in S} y_j^2~\fcar_{\{ |y_j|<\hsig\sqrt{2\log d} \}}\Big)^2
 \nonumber\\
 &\le s^2\esp_\t(\hat\s^4) (2\log d)^2\le C \s^4 s^2 \log^2 d. \label{proof:2n}
\end{align}
Finally, we write the third term as follows
\begin{align}
 & \qquad \esp_\t\Big( \sum_{j\in \tilde{S}\backslash S} y_j^2  \Big)^2= \esp_\t\Big( \sum_{j\not\in S} \s^2\xi_j^2~\fcar_{\{ \s|\xi_j|>\hsig\sqrt{2\log d} \}}\Big)^2 \le 2(A_1 + A_2)
\label{proof:3n}
\end{align}
where
\begin{align*}
\nonumber
& A_1=\esp_\t\Big(\sum_{j=1}^d  \s^2\xi_j^2~\fcar_{\{ \s|\xi_j|>\hsig\sqrt{2\log d} \}}\fcar_{\{\hsig>\sqrt{2}\s\}}\Big)^2, \\
& A_2 = \esp_\t\Big(\sum_{j=1}^d \s^2\xi_j^2~\fcar_{\{\hsig\leq\sqrt{2}\s\}}\Big)^2.
\end{align*} 
Using  \eqref{eq:alpha3} we obtain  
\begin{align}
\label{proof:4n}
A_1&\le \s^4 \esp_\t\Big(\sum_{j=1}^d \xi_j^2~\fcar_{\{|\xi_j|> 2\sqrt{\log d} \}}\Big)^2 %\\ \nonumber
\le  2\s^4 d^2 \esp\big(X^4~\fcar_{\{|X|> 2\sqrt{\log d} \}}\big) \\
&\le C\s^4 (\log d)^{3/2}\nonumber
\end{align} 
where $X\sim \nzeroun$. Next,
\begin{align}
\nonumber
A_2 & \le \s^4 \esp_\t\Big(\sum_{j=1}^d \xi_j^2~\fcar_{\{\hsig\leq\sqrt{2}\s\}}\Big)^2 \le \s^4 d^2\max_{1\le j \le d} \esp_\t \big( \xi_j^4~\fcar_{\{\hsig\leq\sqrt{2}\s\}}\big) .
\end{align} 
Using \eqref{eq:alpha3} we find
\begin{align*}
\nonumber
\esp_\t \big( \xi_j^4~\fcar_{\{\hsig\leq\sqrt{2}\s\}}\big) &\le  \esp_\t \big( \xi_j^4~\fcar_{\{|\xi_j|> 2\sqrt{\log d} \}}\big) + \esp_\t \big( \xi_j^4~\fcar_{\{|\xi_j|\le 2\sqrt{\log d} \}}~\fcar_{\{\hsig\leq\sqrt{2}\s\}}\big)\\
& \le \frac{C}{d^2}(\log d)^{3/2} + 16 (\log d)^2 \prob_\t (\hsig\leq\sqrt{2}\s). 
\end{align*} 
The last two displays and the bound for $\prob_\t (\hsig\leq\sqrt{2}\s)$ from Lemma~\ref{sig} yield 
\begin{align}\label{proof:5n}
A_2 & \le C\s^4(\log d)^{3/2}. 
\end{align} 
Combining \eqref{proof:1n} - \eqref{proof:5n} proves the theorem.

%%%%%%%%%%%%%%%%%%%%%%%

\smallskip

\section{Appendix: Auxiliary lemmas}

\noindent

{\it Proof of Lemma~\ref{th_baraud}.}
We first follow the lines of the proof of Theorem 7 in \cite{CaiLow2004} and then apply a result of \cite{aldous85} in the same spirit as it was done in \cite{Baraud2002}. Let $\varphi_\s$ be a density of normal distribution with mean $0$ and variance~$\s^2$. For $I\in \mathcal{S}(s,d)$, let 
$$g_I(y_1,\ldots,y_d)=\prod_{j=1}^d\varphi_{\s}(y_j-f_j)$$
where $f_j=\s \rho \fcar_{j\in I}$. The density of ${\mathbb P}_{\mu_\rho}$ is 
$$g=\frac{1}{\binom{d}{s}}\sum_{I\in \mathcal{S}(s,d)} g_I$$ 
and we can write
$$
\chi^2({\mathbb P}_{\mu_\rho} ,\prob_0)=\int \Big(\frac{ d {\mathbb P}_{\mu_\rho}}{ d \prob_0}\Big)^2 d \prob_0 -1 = \int \frac{g^2}{f}-1 
$$
where
$f$  is
a density of $n$ i.i.d. normal random variables with mean $0$ and variance~$\s^2$. 
Now, 
$$\int \frac{g^2}{f}= \frac{1}{\binom{d}{s}^2}\sum_{I\in \mathcal{S}(s,d)}
\sum_{I'\in \mathcal{S}(s,d)}\int \frac{g_Ig_{I'}}{f}.$$
It is easy to see that 
$$\int \frac{g_Ig_{I'}}{f}=\exp(\rho^2\card(I\cap I')),$$
which implies
$$
\int \frac{g^2}{f}=\esp \exp(\rho^2J)$$
where $J$ is a random variable with hypergeometric distribution, 
$$\prob(J=j)=\frac{\binom{s}{j}\binom{d-s}{s-j}}{\binom{d}{s}}.$$
As shown in \cite{aldous85},  $J$ coincides in distribution with the conditional expectation $\esp[Z|\mathcal{B}]$ where $Z$ is a binomial random variable with parameters ($s$, $s/d$) and $\mathcal{B}$ is a suitable $\s$-algebra. 
This fact and Jensen's inequality lead to the following bound implying the lemma:
\begin{align*}
\int \frac{g^2}{f}&\leq \esp \exp(\rho^2Z)
=\left(1-\frac{s}{d} + \frac{s}{d} e^{\rho^2}\right)^s.
\end{align*}

\begin{lemma}\label{biais}
Let $y\sim {\mathcal N}(a,\s^2)$ and $\hat{T}=y~\fcar_{\{ |y|>\s\tau \}} $, with $\tau>0$. 
Set $B(a)=\esp(\hat{T})-a$. Then there exists $C>0$ such that 
$$|B(a)|\le C\min(|a|,\s\tau).$$
\end{lemma}
\begin{proof}
Note that $B(a)=\esp(y~\fcar_{\{ |y|\le\s\tau \}} )$, so that
$|B(a)|\le \s\tau.$
Thus, it remains to show that there exists $C>0$ such that $|B(a)|\le C|a|.$ Indeed, if
 $|a|\ge \s$ we have
$$|B(a)|\le \s\esp|X| +|a|\le \Big( \frac{\sqrt{2}}{\pi}+1\Big)|a|,$$
where $X\sim {\mathcal N}(0,1)$.
Finally, if $|a|< \s$ inequality $|B(a)|\le C|a|$ follows from the facts that $B(0)=0$ and 
$|B'(a)|\le 4$ for $|a|< \s$.  
\end{proof}

%%%%%%%%%%%%%%%%
\begin{lemma}\label{biais2}
Let $y\sim {\mathcal N}(a,\s^2)$, $\tilde d\ge 2$, $\tilde x = 2\sqrt{2\log \tilde d}$,
and $|a|\le \s \tilde x /2$. Let $\a$ be such that 
$\esp\Big[(X^2-\a)~\fcar_{\{ |X|> \tilde x \}}\Big] =0$ where $X\sim {\mathcal N}(0,1)$.
Then there exists $C>0$ such that 
$$\Big|\esp\Big[(y^2-\a\s^2)~\fcar_{\{ |y|>\s \tilde x \}}\Big]\Big|\le Ca^2.$$
\end{lemma}
\begin{proof} Using the definition of $\a$ we get
\begin{align*}
\esp\Big[(y^2-\a\s^2)~\fcar_{\{ |y|>\s \tilde x \}}\Big]&=\s^2\esp\Big[(X^2-\a)~\big(\fcar_{\{ |y|>\s \tilde x \}}- \fcar_{\{ |X|> \tilde x \}}\big)\Big]
\\
& \ \ \ \, + 2a\s \esp\Big[X~\fcar_{\{ |y|>\s \tilde x \}}\Big] + a^2 {\mathbf P}(|y|>\s \tilde x).
\end{align*}
Lemma~\ref{biais} implies that  
$$
\Big|\s \esp\Big[X~\fcar_{\{ |y|>\s \tilde x \}}\Big]\Big| = |B(a) + a {\mathbf P}(|y|\le \s \tilde x)| \le C|a|.
$$
Therefore, to finish the proof it remains to show the inequality
\begin{equation}\label{eq:biais1}
\Big|\esp\Big[(X^2-\a)~\big(\fcar_{\{ |y|>\s \tilde x \}}- \fcar_{\{ |X|> \tilde x \}}\big)\Big]\Big|\le C \left(\frac{a}{\s}\right)^2.
\end{equation}
Using the Taylor expansion we obtain
\begin{align*}
&{\mathbf P}(|y|>\s \tilde x)- {\mathbf P}( |X|> \tilde x) = \frac1{\sqrt{2\pi}}\int~\fcar_{\{ |v|> \tilde x \}}\left[e^{-(v-a/\s)^2/2}-e^{-v^2/2}\right]dv
\\
&= \frac1{2\sqrt{2\pi}}
\int~\fcar_{\{ |v|> \tilde x \}}
\left(\frac{a}{\s}\right)^2\Big[\Big(v-\frac{ta}{\s}\Big)^2 -1 \Big] e^{-(v-ta/\s)^2/2}
dv
\end{align*}
where $0\le t\le1$.  By the assumption on $a$, on the set $\{ |v|> \tilde x \}$ we have $|v|/2 \le |v-ta/\s| \le 3 |v|/2$. Hence,
\begin{align}\label{eq:biais2}
\a |{\mathbf P}(|y|>\s \tilde x)- {\mathbf P}( |X|> \tilde x)| &\le 
\frac{\a}{2\sqrt{2\pi}}\left(\frac{a}{\s}\right)^2
\int~\fcar_{\{ |v|> \tilde x \}}
(9v^2/4 +1) e^{-v^2/8}
dv
\\ \nonumber
& \le C \left(\frac{a}{\s}\right)^2 \frac{(\log \tilde d)^{3/2}}{\tilde d} \le C \left(\frac{a}{\s}\right)^2 \,
\end{align}
where we have used Lemma~\ref{alpha} and the facts that $\tilde x \ge 2\sqrt{2\log 2}, \tilde d\ge 2$, and  $\a\le 40\log \tilde d$ (which is proved analogously to \eqref{eq:alpha}). Similarly,
\begin{align*}
&\esp\Big[X^2~\big(\fcar_{\{ |y|>\s \tilde x \}}- \fcar_{\{ |X|> \tilde x \}}\big)\Big]
= \frac1{\sqrt{2\pi}}\int\fcar_{\{ |v|> \tilde x \}}\left[(v-a/\s)^2e^{-(v-a/\s)^2/2}-v^2e^{-v^2/2}\right]dv,
\end{align*}
and from the Taylor expansion of $v^2e^{-v^2/2}$ and Lemma~\ref{alpha} we deduce, as in \eqref{eq:biais2}, that
\begin{align*}%\label{eq:biais3}
\esp\Big[X^2~\big(\fcar_{\{ |y|>\s \tilde x \}}- \fcar_{\{ |X|> \tilde x \}}\big)\Big]
&\le
\frac1{2\sqrt{2\pi}}\left(\frac{a}{\s}\right)^2
\int~\fcar_{\{ |v|> \tilde x \}}
\Big[\Big(\frac{3v}{2}\Big)^4 + 5\Big(\frac{3v}{2}\Big)^2 + 2\Big] e^{-v^2/8}
dv
\\ \nonumber
& \le C \left(\frac{a}{\s}\right)^2 \frac{(\log \tilde d)^{3/2}}{\tilde d} \le C \left(\frac{a}{\s}\right)^2 \,
\end{align*}
(we have used that $(v^2e^{-v^2/2})'' =(v^4-5v^2+2)e^{-v^2/2}$). Combining the last display with  \eqref{eq:biais2} we obtain  \eqref{eq:biais1} and thus the lemma.
\end{proof}

%%%%%%%%%%%%%%
 \begin{lemma}\label{sig} For any $\t$ such that $\|\t\|_0\le \sqrt{d}$ we have
 \begin{equation}\label{eq:sig1}
\esp_\t(\hsig^2)\leq 9\s^2, \qquad \esp_\t(\hsig^4)\leq C\s^4,
\end{equation} 
and 
\begin{equation}\label{eq:sig2}
\prob_\t(\hsig\leq \s)\leq Cd\exp(-\sqrt{d}/C) 
\end{equation}
for some absolute constant $C>0$.
 \end{lemma}
 \begin{proof} Since $\|\t\|_0\le \sqrt{d}$ we have
 $$\hsig^2\leq \frac{9}{d}\sum_{j=1}^{ d-\|\t\|_0} y_{(j)}^2.$$
Denote by $F$ the set of indices $i$ corresponding to the $d-\|\t\|_0$ smallest values $y_i^2$. Then 
\begin{equation*}
\sum_{j=1}^{d-\|\t\|_0} y^2_{(j)} =\sum_{i\in F} y^2_i = \s^2\sum_{i \in S^c} \xi_i^2+ \sum_{i\in S \cap  F} y^2_i- \s^2\sum_{i\in S^c \cap  F^c} \xi_i^2
\end{equation*}
where $S=\{j : \t_j\neq 0\}$.
For  any $i \in S \cap F$ and  any $j \in S^c\cap F^c$, we have
$$y^2_i\le \s^2 \xi^2_j .$$ 
Furthermore, $\card(S \cap  F)=\card(S^c \cap  F^c)$.
Therefore, $$\hsig^2\leq \frac{9\s^2}{d}\sum_{i\in S^c} \xi_i^2.$$ 
This implies \eqref{eq:sig1}.
 We now prove \eqref{eq:sig2}. Let
 $G$ be the set of indices $i$ corresponding to the $\lfloor d-\sqrt{d}\rfloor$ smallest~$y_i^2$. Here, $\lfloor x \rfloor$ denotes the largest integer less than or equal to $x$. Then we have
$$\sum_{j\le d-\sqrt{d}} y_{(j)}^2= \sum_{i\in G} y^2_i\geq \s^2\sum_{i\in S^c\cap G}\xi_i^2\geq \s^2\sum_{i\in S^c}\xi_i^2 -2\sqrt{d} \,\s^2\max_{i\in S^c} \xi_i^2,
$$
where we have used that $\card(G^c) \le 2\sqrt{d}$.
This implies:
$$\hsig^2\geq \frac{9\s^2}{d} \sum_{i \in S^c}\xi_i^2 -\frac{18\s^2}{\sqrt{d}} \max_{i \in S^c} \xi_i^2.$$
Thus,
\begin{align}\nonumber
\prob_\t(\hsig\leq \sqrt{2} \s)&\leq \prob\Big(9\s^2 \sum_{i\in S^c}\xi_i^2 -18\sqrt{d}\s^2\max_{i\in S^c} \xi_i^2\leq 2d\s^2\Big)\\
&\leq \prob\Big(9\sum_{i\in S^c}\xi_i^2 \leq 3 d\Big)+\prob\Big(18\max_{i\in S^c}\xi_i^2 \geq \sqrt{d} \Big).
\label{2termes}
\end{align}
The first term on the right hand side of \eqref{2termes} satisfies
$$\prob\Big(3\sum_{i\in S^c}\xi_i^2 \leq d\Big)\leq \prob\left(U_D-D\leq -2d/3 + \sqrt{d}\right)$$
where $D=\card (S^c)$, and $U_D$ is a $\chi^2$ random variable with $D$ degrees of freedom. A standard bound on the tails of $\chi^2$ random variables (see, e.g. \cite{LaurentMassart2000}) yields
$$\prob(U_D-D\leq -t ) \leq \exp(-t^2/(4D)), \qquad \forall \ t>0.$$
Thus, for $d>2$, we obtain
$$\prob\Big(3\sum_{i\in S^c}\xi_i^2 \leq d\Big)\leq \exp(-(2d/3 - \sqrt{d})^2/(4D)) \le  \exp(- d/C)$$
where $C>0$ is an absolute constant. 
Finally,  the  second term on the right hand side of \eqref{2termes} satisfies
\begin{align*}
\prob\left(\max_{i\in S^c}\xi_i^2 \geq \frac{\sqrt{d}}{18} \right)
&\leq d\exp\left(-\frac{\sqrt{d}}{36}\right) 
\end{align*}
in view of \eqref{eq:alpha1}. Plugging the last two displays in \eqref{2termes} we obtain \eqref{eq:sig2}.
\end{proof}

\section*{Acknowledgement}
We would like to thank Nicolas Verzelen for remarks on the text that helped to improve the presentation.
The work of A.B.Tsybakov was supported by GENES and by the French National Research Agency (ANR) under the grants 
IPANEMA (ANR-13-BSH1-0004-02), Labex ECODEC (ANR - 11-LABEX-0047), and ANR -11- IDEX-0003-02.

%\bibliographystyle{imsart-nameyear}
%\bibliography{bibli}

%\end{document}

\end{document}